\documentclass[12pt]{amsart}
\usepackage{url, 
	amssymb,setspace, mathrsfs,fontenc, comment}
\usepackage[alphabetic]{amsrefs}
\usepackage{fullpage} 
\usepackage{color}
\usepackage{tikz-cd}
\usepackage[all]{xy}
\usepackage{amsmath}

\usepackage{url, 
	amssymb,setspace, mathrsfs,fontenc}
\usepackage{amsrefs}
\usepackage{graphicx}
\usepackage[mathcal]{euscript}
\usepackage{verbatim}
\usepackage{hyperref}
\hypersetup{backref=true}
\usepackage{mathtools}
\usepackage{tikz}
\usetikzlibrary{chains}

\tikzset{node distance=2em, ch/.style={circle,draw,on chain,inner sep=2pt},chj/.style={ch,join},every path/.style={shorten >=4pt,shorten <=4pt},line width=1pt,baseline=-1ex}

\newtheorem{thm}{Theorem}
\newtheorem{lem}[thm]{Lemma}
\newtheorem{prop}[thm]{Proposition}

\newtheorem{cor}[thm]{Corollary}
\newtheorem{defe}[thm]{Definition}

\theoremstyle{remark}
\newtheorem{rem}[thm]{Remark}

\DefineSimpleKey{bib}{myurl}
\newcommand\myurl[1]{\url{#1}}
\BibSpec{webpage}{
	+{}{\PrintAuthors} {author}
	+{,}{ \textit} {title}
	+{}{ \parenthesize} {date}
	+{,}{ \myurl} {myurl}
}
\usepackage{arydshln}

\newcommand{\nc}{\newcommand}

\nc{\ssec}{\subsection}

\nc{\on}{\operatorname}

\nc{\sE}{\mathscr{E}}
\nc{\sF}{\mathscr{F}}
\nc{\sL}{\mathscr{L}}
\nc{\sD}{\mathscr{D}}
\nc{\sA}{\mathscr{A}}

\nc{\cC}{\mathcal{C}}
\nc{\cG}{\mathcal{G}}
\nc{\cV}{\mathcal{V}}
\nc{\CB}{\mathcal{B}}
\nc {\K}{\mathcal{K}}

\nc{\cE} {\mathcal{E}}
\nc{\Kl}{\mathrm{Kl}}
\nc{\cO}{\mathcal{O}}
\nc{\cF}{\mathcal{F}}
\nc{\cZ}{\mathcal{Z}}
\nc{\bcZ}{\overline{\mathcal{Z}}}
\nc{\bcB}{\overline{\mathcal{B}}}
\nc{\cD}{\mathcal{D}}
\nc{\cDt}{\mathcal{D}^\times}
\nc{\cH}{\mathcal{H}}
\nc{\bZ}{\mathbb{Z}}
\nc{\bH}{\mathbb{H}}
\nc{\bQ}{\mathbb{Q}}
\nc{\bR}{\mathbb{R}}
\nc{\bC}{\mathbb{C}}
\nc{\bQl}{\overline{\mathbb{Q}}_\ell}
\nc{\bQlt}{\bQl^\times} 
\nc{\FG}{\mathrm{FG}}
\nc{\dR}{\mathrm{dR}}
\nc{\dv}{\dot{v}}
\nc{\du}{\dot{u}}
\nc{\bbR}{\mathbb{R}}

\nc{\uG}{\underline{G}}
\nc{\uc}{\underline{c}}
\nc{\uu}{\underline{u}}
\nc{\cU}{\mathcal{U}}
\nc{\rat}{\mathrm{rat}}
\nc{\Hyp}{\mathrm{Hyp}}
\nc{\Lie}{\mathrm{Lie}}
\nc{\ctheta}{\check{\theta}}
\nc{\nil}{\mathrm{nil}}
\nc{\bLX}{\overline{LX}}
\nc{\bOmega}{\overline{\Omega}}
\nc{\tOmega}{\widetilde{\Omega}}

\nc{\fF}{\mathfrak{F}}
\nc{\fB}{\mathfrak{B}}
\nc{\fZ}{\mathfrak{Z}}
\nc{\fx}{\mathfrak{x}}
\nc{\fy}{\mathfrak{y}}
\nc{\fb}{\mathfrak{b}}
\nc{\fk}{\mathfrak{k}}
\nc{\fI}{\mathfrak{i}}
\nc{\fj}{\mathfrak{j}}
\nc{\fg}{\mathfrak{g}}
\nc{\fu}{\mathfrak{u}}
\nc{\fl}{\mathfrak{l}}
\nc{\fn}{\mathfrak{n}}
\nc{\cP}{\mathcal{P}}
\nc{\cQ}{\mathcal{Q}}
\nc{\ft}{\mathfrak{t}}
\nc{\fz}{\mathfrak{z}}
\nc{\fc}{\mathfrak{c}}
\nc{\cfc}{\check{\mathfrak{c}}}
\nc{\fh}{\mathfrak{h}}
\nc{\fp}{\mathfrak{p}}
\nc{\cfp}{\mathring{\mathfrak{p}}}
\nc{\bone}{\mathbf{1}}
\nc{\tg}{\mathtt{g}}
\nc{\hfg}{\widehat{\fg}}
\nc{\ch}{\check{\fh}}
\nc{\hP}{\hat{P}}
\nc{\hg}{\widehat{\mathfrak{g}}}
\nc{\gO}{\mathfrak{g}[\![t]\!]}
\nc{\Ug}{\widehat{U}(\mathfrak{g})}
\nc{\dl}{/\!\!/}

\nc{\bGm}{\mathbb{G}_m}
\nc{\bGa}{\mathbb{G}_a}
\nc{\bL}{\mathbf{L}}
\nc{\bK}{\mathbf{K}}
\nc{\bJ}{\mathbf{J}}
\nc{\bI}{\mathbf{I}}
\nc{\bV}{\mathbb{V}}
\nc{\bM}{\mathbb{M}}
\nc{\bP}{\mathbb{P}}
\nc{\bA}{\mathbb{A}}
\nc{\bN}{\mathbb{N}}

\nc {\Q}{\mathrm{Q}}
\nc{\diag}{\mathrm{diag}}
\nc{\diff}{\mathrm{diff}}
\nc{\ev}{\mathrm{ev}}
\nc{\Res}{\mathrm{Res}}
\nc{\Fl}{\mathcal{F}\ell}
\nc{\Ad}{\mathrm{Ad}}
\nc{\ad}{\mathrm{ad}}
\nc{\pr}{\mathrm{pr}}
\nc{\Sl}{\mathfrak{sl}}
\nc{\gl}{\mathfrak{gl}}
\nc{\ra}{\rightarrow}
\nc{\tra}{\twoheadrightarrow}
\nc{\hra}{\hookrightarrow}
\nc{\quo}{\mathopen{ /\!/}}
\nc{\GL}{\mathrm{GL}}
\nc{\SL}{\mathrm{SL}}
\nc{\Sp}{\mathrm{Sp}}
\nc{\SO}{\mathrm{SO}}
\nc{\so}{\mathfrak{so}}
\nc{\PGL}{\mathrm{PGL}}
\nc{\Bun}{\mathrm{Bun}}
\nc{\supp}{\mathrm{supp}}
\nc{\bgamma}{\bar{\gamma}}
\nc{\ab}{\mathrm{ab}}
\nc{\td}{\mathrm{d}}
\nc{\Ht}{\mathrm{ht}}
\nc{\tX}{\tilde{X}}

\nc         {\rar}[1]       {\stackrel{#1}{\longrightarrow}}

\nc{\fa}{\mathfrak{a}}
\nc{\Hit}{\mathrm{Hit}}

\nc{\RS}{\mathrm{RS}}
\nc{\Loc}{\mathrm{Loc}}
\nc{\tLoc}{\widetilde{\mathrm{Loc}}}
\nc{\reg}{\mathrm{reg}}
\nc{\im}{\mathrm{Im}}

\nc{\tp}{\mathfrak{p}}
\nc{\cA}{\mathcal{A}}
\nc{\cY}{\mathcal{Y}}

\nc{\opp}{\mathrm{opp}}
\nc{\Ind}{\mathrm{Ind}}
\nc{\sAn}{\mathrm{can}}
\nc{\Lg}{\check{\fg}}
\nc{\cDelta}{\check{\Delta}}
\nc{\cPhi}{\check{\Phi}}
\nc{\LV}{\check{V}}
\nc{\Lh}{\check{h}}
\nc{\LG}{\check{G}}
\nc{\cT}{\check{T}}
\nc{\ct}{\check{\ft}}
\nc{\cB}{\check{B}}
\nc{\cb}{\check{\fb}}
\nc{\cN}{\check{N}}
\nc{\sN}{\mathcal{N}}
\nc{\cn}{\check{\fn}}
\nc{\Spec}{\mathrm{Spec}}
\nc{\End}{\mathrm{End}}
\nc{\crho}{\check{\rho}}
\nc{\clambda}{\check{\lambda}}
\nc{\rX}{\mathring{X}}
\nc{\ru}{\mathring{u}}

\nc{\sW}{\mathscr{W}}
\nc{\sH}{\mathscr{H}}
\nc{\sV}{\mathscr{V}}
\nc{\geom}{\mathrm{geom}}
\nc{\Irr}{\mathrm{Irr}}
\nc{\fm}{\mathfrak{m}}
\nc{\aff}{\mathrm{aff}}
\nc{\Aut}{\mathrm{Aut}}
\nc{\cJ}{\mathcal{J}}
\nc{\fs}{\mathfrak{s}}
\nc{\Stab}{\mathrm{Stab}}
\nc{\st}{\mathrm{st}}
\nc{\tw}{{\widetilde{w}}}
\nc{\gen}{\mathrm{gen}}
\nc{\genn}{\mathrm{genn}}
\nc{\sss}{\mathrm{ss}}
\nc{\fsp}{\mathfrak{sp}}
\nc{\Hom}{\mathrm{Hom}}
\nc{\bm}{\mathbf{m}}
\nc{\HG}{\mathcal{HG}}
\nc{\Gal}{\mathrm{Gal}}
\nc{\Sym}{\mathrm{Sym}}
\nc{\rank}{\mathrm{rank}}

\nc{\calX}{\mathcal{X}}
\nc{\tP}{\mathtt{P}}
\nc{\tL}{\mathtt{L}}
\nc{\tU}{\mathtt{U}}

\nc{\tW}{\widetilde{W}}
\nc{\tdb}{\tilde{b}}
\nc{\tdd}{\tilde{d}}
\nc{\tv}{\tilde{v}}
\nc{\Hk}{\on{Hk}}
\nc{\cL}{\mathcal{L}}
\nc{\talpha}{\widetilde{\alpha}}
\nc{\tQ}{{\widetilde{Q}}}
\nc{\ochi}{\overline{\chi}}
\nc{\tdelta}{\widetilde{\Delta}}
\nc{\wt}{\mathrm{wt}}
\nc{\fQ}{\mathfrak{Q}}
\nc{\bbP}{\mathbb{P}}

\nc{\inv}{\mathrm{inv}}
\nc{\Rep}{\mathrm{Rep}}
\nc{\Conn}{\mathrm{Conn}}
\nc{\Hecke}{\mathrm{Hecke}}
\nc{\Gr}{\mathrm{Gr}}
\nc{\GR}{\mathrm{GR}}
\nc{\IC}{\mathrm{IC}}
\nc{\Std}{\mathrm{Std}} 
\nc{\Db}{\mathrm{D}^{\mathrm{b}}}
\nc{\tr}{\mathrm{tr}}
\nc{\gr}{\mathrm{gr}}
\nc{\tmin}{\mathrm{min}}
\nc{\Fun}{\mathrm{Fun}~}

\nc{\bbA}{\mathbb{A}}
\nc{\mO}{\mathrm{O}}

\newcommand{\quash}[1]{}

\AtEndDocument{\bigskip{\footnotesize

\textsc{Tsao-Hsien Chen, School of Mathematics, University of Minnesota, Twin cities, Minneapolis, MN 55455 } \par
\textit{E-mail address}: \texttt{chenth@umn.edu} \par
		
\textsc{Lingfei Yi, Shanghai Center for Mathematical Sciences, Fudan University, Shanghai 200438, China} \par
\textit{E-mail address}: \texttt{yilingfei@fudan.edu.cn} \par
}}

\begin{document} 
\renewcommand{\thepart}{\Roman{part}}

\renewcommand{\partname}{\hspace*{20mm} Part}

\title{Matsuki duality for loop groups} 
\date{\today}

 \author{Tsao-Hsien Chen and Lingfei Yi}

\dedicatory{To George Lusztig
with deepest admiration}

\maketitle
\begin{abstract}
    We establish  versions of Matsuki duality for loop groups. The main result is a bijection between  symmetric loop group orbits and  real 
     polynomial loop group orbits on the affine Grassmannians or  affine flag varieties. Along the way we obtain orbit  parametrizations  and make connections with vector bundles on real and twistor-$\mathbb P^1$ and Kottwitz sets .
    
\end{abstract}

\tableofcontents
\section{Introduction}
Let $G$ be a complex connected reductive group.
Let $G_\bR$ be a real from of $G$ and let $K$ be a symmetric subgroup of $G$ such that $K_\bR=G_\bR\cap K$
is a maximal compact subgroup of both $G_\bR$ and $K$.
The celebrated Matsuki duality for $G$  is a bijection 
\[K\backslash\mathcal B\longleftrightarrow G_\bR\backslash\mathcal{B} \]
between 
 $K$-orbits and $G_\bR$-orbits on the flag variety $\mathcal B$ of $G$ with  remarkable properties \cite{M}.
In this paper we establish a version of Matsuki duality for loop groups.
The 
main result is a bijection 
\[
  K(\!(t)\!)\backslash\Fl\longleftrightarrow LG_\bR\backslash\Fl  
\]
between 
$K(\!(t)\!)$-orbits and 
$LG_\bR$-orbits on the affine flag variety $\Fl$ with several remarkable properties. Here 
$K(\!(t)\!)$ is the loop group of $K$ and $LG_\bR$ is the polynomial loop group of $G_\bR$ 
consisting of maps that
take the unit circle $S^1\subset\mathbb C$ to $G_\bR\subset G$.
We also establish a version of Matsuki duality for 
twisted loop groups orbits 
and make connections with vector bundles on the twistor-$\mathbb P^1$ and Kottwitz sets $B(G_\bbR)$ for $G_\bR$. 

Our approach is entirely combinatorial and 
is inspired by 
the work \cite{N} of Nadler on Matsuki duality for affine Grassmannians and 
our previous work \cite{CY} on singularities of $ K(\!(t)\!)$-orbit closures on $\Fl$.
The method of proof
treats simultaneously all
the real forms that are in the same  pure inner class.
Along the way, we obtain  parametrizations for both 
$K(\!(t)\!)$ and $LG_\bR$-orbits
and the equivariant local systems on them.
In a forthcoming work \cite{CNY}, we will provide a 
Morse-theoretic interpretation and refinement 
of the Matuski duality and use it to
upgrade the Matuski duality to an equivalence of equivariant derived categories, extending the works  \cite{MUV} and \cite{CN} in the setting of 
flag varieties and affine Grassmannians.
Such an equivalence is a key ingredient to relate geometric Langlands for real groups and relative Langlands duality, see \cite{C,CN}.

We now describe the paper in more details.
% The bijection is a key ingredient in geometric approach  to representation theory of real groups, see, e.g., \cite{MUV}.%

\subsection{Main results}
\subsubsection{}
We fix $\epsilon\in\{1,-1\}$.
Let $F=\bC(\!(t)\!)$ and  $\cO=\bC[\![t]\!]$ be the formal Laurent series and Taylor series rings.
Let $G(F)$ and $G(\cO)$ be the formal loop group and formal arc group of $G$.
Let $G[t,t^{-1}]\subset G(F)$
be the polynomial loop group and 
$G[t]\subset G(\cO)$, $G[t^{-1}]\subset G(F)$
be the polynomial arc group and the negative polynomial arc group.
Let $I\subset G(\cO)$ be an Iwahori subgroup associated to a Borel subgroup $B$ and we denote by $I_p=I\cap G[t]$.
Let $\Gr=G(F)/G(\cO)$ and $\Fl=G(F)/I$ be the affine Grassmannian and affine flag variety respectively

Let $\eta_0$ be a conjugation of $G$ such that $G_\bR=G^{\eta_0}$
and $\theta_0$ be a complex Cartan involution on $G$ such that $K=G^{\theta_0}$.
Let $G_c=G^{\eta_0\circ\theta_0}=G^{\theta_0\circ\eta_0}$ be a compact real from of $G$ and let $T_c\subset G_c$
be a maximal torus.
Consider  the twisted involution on $G(F)$:
\begin{equation}
	\theta(\gamma)(t)=\theta_0(\gamma(\epsilon t)).
\end{equation}
We denote by $G(F)^\theta$ 
the $\theta$ fixed points subgroup 
of $G(F)$ 
and the $\theta$ anti-fixed points set  by
\[G(F)^{\on{inv}\circ\theta}=\{\gamma\in G(F)|\theta(\gamma)=\gamma^{-1}\}.\]
The loop group $G(F)$ acts on $G(F)^{\on{inv}\circ\theta}$  by the $\theta$-conjugation action 
$
h\cdot_{\theta} \gamma=h\gamma\theta(h)^{-1}$.
Consider the twisted involution on $G[t,t^{-1}]$:
\begin{equation}
	\eta(\gamma)(t)=\eta_0(\gamma(\epsilon \bar t^{-1})).
\end{equation}
Denote by $G[t,t^{-1}]^\eta$
the $\eta$ fixed point subgroup of $G[t,t^{-1}]$
and the $\eta$ anti fixed points 
 by 
\[G[t,t^{-1}]^{\on{inv}\circ\eta}=\{\gamma\in G[t,t^{-1}]|\eta(\gamma)=\gamma^{-1}\}.\]
Then $G[t,t^{-1}]^{\on{inv}\circ\eta}$ is stable under 
the $\eta$-conjugation action on $G[t,t^{-1}]$ by 
$h\cdot_{\eta}\gamma=g\gamma\eta(\gamma)^{-1}$.
Note that in the un-twisted setting, that is, when $\epsilon=1$, we have 
$G(F)^{\theta}=K(F)$ and $G[t,t^{-1}]^\eta=LG_\bR$
consisting of maps 
$\gamma:\bC^\times\to G$
that take the unit circle $S^1\subset\mathbb C$ to $G_\bR\subset G$.
\subsubsection{}
The first main result of the paper is a version of Matsuki duality for loop groups: 
\begin{thm}\label{intro: main 1}\mbox{}
\begin{itemize}
    \item [(i)]
    There is a bijection
	\[
	G(\cO)\backslash_\theta G(F)^{\inv\circ\theta}
	\simeq
	G[t]\backslash_\eta(G[t,t^{-1}])^{\inv\circ\eta}
	\]
   between $G(\cO)$-orbits
   on $G(F)^{\on{inv}\circ\theta}$
   and $G[t]$-orbits on $G[t,t^{-1}]^{\on{inv}\circ\eta}$
   such that the intersection of the corresponding 
   $G(\cO)$-orbit and $G[t]$-orbit
   is a single  $G_c$-orbit.
   Furthermore, $G(\cO)$-equivariant local systems on a $G(\cO)$-orbit corresponds bijectively to $G[t]$-equivariant local systems on the corresponding $G[t]$-orbit.
   \item[(ii)]
   There is a bijection
	\[
	I\backslash_\theta G(F)^{\inv\circ\theta}
	\simeq
	I_p\backslash_\eta(G[t,t^{-1}])^{\inv\circ\eta}
	\]
   between $I$-orbits
   on $G(F)^{\on{inv}\circ\theta}$
   and $I_p$-orbits on $G[t,t^{-1}]^{\on{inv}\circ\eta}$
   such that the intersection of the corresponding 
   $I$-orbit and $I_p$-orbit
   is a single $T_c$-orbit.
   Furthermore, $I$-equivariant local systems on an $I$-orbit
   corresponds bijectively to $I_p$-equivariant local systems on the corresponding $I_p$-orbit.
\end{itemize}
\end{thm}

Theorem \ref{intro: main 1} is the combination of  Theorem \ref{t:affine Matsuki} and Theorem \ref{t:Iwahori Matsuki}.
We refer to Section \ref{spherical orbits para}
and \ref{iwahori orbits para} for a more
detailed explanation of the statement including 
orbits parametrizations and 
the characterization of the bijections.

The anti-fixed points set $G(F)^{\on{inv}\circ\theta}$ 
(resp. $G[t,t^{-1}]^{\on{inv}\circ\eta}$)
is in fact a finite union of $G(F)$-orbits 
(resp. $G[t]$-orbits) 
and, by restricting to the orbit through the based points, we obtain 
a Matsuki duality for $\Gr$
and $\Fl$:

\begin{thm}\label{intro: main 2}\mbox{}
\begin{itemize}
    \item[(i)] 
    There is a bijection
	\[
	 G(F)^{\theta}\backslash\Gr
	\simeq
G[t,t^{-1}]^{\eta}\backslash\Gr
	\]
   between $G(F)^{\theta}$-orbits
   and $G[t,t^{-1}]^{\eta}$-orbits on $\Gr$
   such that the intersection of the corresponding 
   $G(F)^{\theta}$-orbit and $G[t,t^{-1}]^{\eta}$-orbit 
   is a single $G(F)^\theta\cap G[t,t^{-1}]^\eta$-orbit.
   Furthermore, $G(F)^\theta$-equivariant local systems on 
   a $G(F)^{\theta}$-orbit corresponds bijectively to $G[t,t^{-1}]^{\eta}$-equivariant local systems on the
   corresponding $G[t,t^{-1}]^{\eta}$-orbit.
   \item[(ii)]
   There is a bijection
	\[
	 G(F)^{\theta}\backslash\Fl
	\simeq
G[t,t^{-1}]^{\eta}\backslash\Fl
	\]
   between $G(F)^{\theta}$-orbits
   and $G[t,t^{-1}]^{\eta}$-orbits on $\Fl$
   such that the intersection of the corresponding 
   $G(F)^{\theta}$-orbit and $G[t,t^{-1}]^{\eta}$-orbit 
   is a single  $G(F)^\theta\cap G[t,t^{-1}]^\eta$-orbit.
   Furthermore, $G(F)^\theta$-equivariant local systems on a $G(F)^{\theta}$-orbit corresponds bijectively to 
   $G[t,t^{-1}]^{\eta}$-equivariant local systems 
   on the corresponding $G[t,t^{-1}]^{\eta}$-orbit.
\end{itemize}
\end{thm}

Theorem \ref{intro: main 2} is the combination of Theorem \ref{Gr} and Theorem \ref{Fl}.
We refer to Section \ref{spherical orbits para}
and \ref{iwahori orbits para} for a more
detailed explanation
including the compatibility with Theorem \ref{intro: main 1}.
In the un-twisted case $\epsilon=1$,  Theorem \ref{intro: main 2}.(i) specializes to the Matsuki duality for $\Gr$ in \cite{N}.
The case of affine flag varieties $\Fl$ and the twisted setting $\epsilon=-1$ are new (even in the case of $\Gr$). 

\begin{rem}
In \cite{CNY}, we will 
show that that all the correspondences in 
    Theorem \ref{intro: main 1} and Theorem \ref{intro: main 2}
    can be characterized by the property that
    the intersection of the corresponding orbits is a single orbit of
    $G_c,T_c,G(F)^\theta\cap G[t,t^{-1}]^\eta$ respectively.
    For the case of Theorem \ref{intro: main 2}.(ii) and $\epsilon=1$,
    this is known in \cite[Theorem 1.2]{N}.

\end{rem}

\subsubsection{}
Finally, we make connectoins to 
vector bundles on real and twistor-$\mathbb P^1$. 
Let $\mathbb P^1$ be the complex projective line
with coordinate $t$.
Consider the conjugation on $\mathbb P^1$
\[\eta:\mathbb P^1\to \mathbb P^1,\ \ \eta(t)=
\epsilon\bar{t}^{-1}.\]
Denote by 
$\mathbb P^1_{\eta}$ the scheme over $\mathbb R$ that is obtained by descending 
the complex projective line $\mathbb P^1$ to $\mathbb R$ via the conjugation $\eta$.
If $\epsilon=1$, then $\mathbb P^1_\eta\cong\mathbb P^1_\mathbb R$ is the real projective with real points $\mathbb P^1_\mathbb R(\mathbb R)=S^1=\{t\in\mathbb{C}||t|=1\}$ the unit circle.
If $\epsilon=-1$, then  $\mathbb P^1_\eta\cong\widetilde{\mathbb P}^1_\mathbb R\cong\on{Proj}(\bbR[x,y,z]/(x^2+y^2+z^2))$ is the 
so called twistor-$\mathbb P^1$.
Denote by 
\[\Bun_G(\mathbb P^1)_{\eta}(\mathbb R)\]
the set of isomorphism classes of 
$G_\bbR$-bundles on $\mathbb P^1_\eta$
and 
\[\Bun_G(\mathbb P^1,0,\infty)_{\eta}(\mathbb R)\]
 the set of isomorphism classes of $G_\bR$-bundles on $\mathbb P^1_\eta$
together with  $B$-reductions  at $[0]=[\infty]\in\mathbb P^1(\bC)/\eta$.
The uniformization of real bundles on $\mathbb P^1_\eta$ provides an identification of $\Bun_G(\mathbb P^1)_{\eta}(\mathbb R)$ and 
$\Bun_G(\mathbb P^1,0,\infty)_{\eta}(\mathbb R)$ with the right hand sides of Theorem \ref{intro: main 1},
and hence we obtain the following 
Matsuki duality between 
spherical and Iwahori orbits on 
$G(F)^{\on{inv}\circ\theta}$
and $G_\bR$-bundles on $\mathbb P^1_\eta$:

\begin{thm}\label{intro: main 3}\mbox{}
\begin{itemize}
    \item[(i)]  
    There is a bijection
	\[
	G(\cO)\backslash_\theta G(F)^{\inv\circ\theta}
	\simeq\Bun_G(\mathbb P^1)_{\eta}(\mathbb R)
	\]
   between $G(\cO)$-orbits
   on $G(F)^{\on{inv}\circ\theta}$
   and isomorphism classes of $G_\bR$-bundle 
   on $\mathbb P^1_\eta$.
   \item[(ii)]  
   There is a bijection
	\[
	I\backslash_\theta G(F)^{\inv\circ\theta}
	\simeq\Bun_G(\mathbb P^1,0,\infty)_{\eta}(\mathbb R)
	\]
   between $I$-orbits
   on $G(F)^{\on{inv}\circ\theta}$
   and isomorphism classes of $G_\bR$-bundles 
   on $\mathbb P^1_\eta$ together with $B$-reductions at $[0]=[\infty]\in\mathbb P^1(\bC)/\eta$.
   \end{itemize}
\end{thm}
Theorem \ref{intro: main 3} is the combination of  Proposition \ref{Unif spherical} and Proposition \ref{unif Iwahori}.
We refer to Section \ref{Real bundles}
 for a more
detailed explanation of the statement including 
the connections with Kottiwtz set $B(G_\bR)$ for $G_\bR$, Section \ref{Kottwitz}.
Theorem \ref{intro: main 3}
is a key ingredient 
connecting relative and real Langlands duality, see \cite{C} and \cite{CN}. 

\begin{rem}
In the case $G_\bR=\GL_n(\mathbb R$),
the set 
$\Bun_G(\mathbb P^1)_{\eta}(\mathbb R)$ 
is what  Simpson  called 
circular or antipodal  real twistor structures
depending on whether $\epsilon=1$ or $-1$, see  \cite[Section 2]{S}.
\end{rem}

\subsection{Derived equivalences}
In a forthcoming work \cite{CNY}, we will provide a Morse-theoretic interpretation and refinement of the Matsuki dualities
 and use it to obtain a lift of the main results   
to an equivalence of derived categories of sheaves,
extending the previous work 
in  \cite{MUV} and  \cite{CN} in the settings of flag varieties $\mathcal B$ and affine Grassmannians $\Gr$ with  $\epsilon=1$.

\begin{rem}
In the un-twisted setting of $\Gr$
    with $\epsilon=1$ in \cite{N} and \cite{CN}, there is a
    remarkable connection between the Matsuki duality and
the spherical orbits on the real affine Grassmannian $\Gr_\bR=G_\bR(\bR(\!(t)\!))/G_\bR(\bR[\![t]\!])$.
Using such a connection one can give another construction
of the Matsuki duality 
and its lifting to derived equivalences
without the Morse-theoretic interpretation.
Such a connection to real affine Grassmannians is not available in the setting of affine flag varieites $\Fl$
and the twsited setting $\epsilon=-1$.

\end{rem}

\subsection{Organization}
In Section \ref{loop group}, we recall basic notions in real groups and discuss
involutions on loop groups. 
In Section \ref{spherical orbits para}, 
we study 
spherical orbits parametrizations and establish
Matsuki duality for spherical orbits.
In Section \ref{iwahori orbits para}, we study 
Iwahori orbits parametrizations and 
establish Matsuki duality for Iwahori orbits.
In Section \ref{Real bundles}, 
we study connections with 
bundles on real and twsitor-$\mathbb{P}^1$
and Kottwitz sets.
In Section \ref{examples}, we discuss examples for real froms of $\GL_2$.

\subsection{Acknowledgements} 
The paper and our previous work \cite{CY} are inspired by Lusztig's 
influential works \cite{LV} and \cite{KL} and  his questions in \cite{L}. 
We dedicate this work to him with our deepest gratitude.

The authors also would like to thank Dougal Davis and 
David Nadler for useful discussions.
The research of
T.-H.~Chen is supported by NSF grant DMS-2143722.
Lingfei Yi is supported by the National Key R\&D Program of China No. 2024YFA1014600 and Grant No. JIH1414062Y of Fudan University. 

\section{Loop groups}\label{loop group}
\subsection{Group data}
Let $G$ be a complex connected reductive group.
Let $\eta_0$ be a conjugation on $G$ with real form
$G_\bbR=G^{\eta_0}$.
Let $\eta_{c,0}$ be a compact conjugation of $G$
such that $\theta_0=\eta_0\circ\eta_{c,0}=\eta_{c,0}\circ\eta_0$ is an involution on $G$.
Let $G_c=G^{\eta_{c,0}}$ and $K=G^{\theta_0}$ 
 be the associated  compact real form and symmetric subgroup.
Let $\fg=\fk\oplus\fp_c$ the eigenspace decomposition
of $\eta_{c,0}$, where $\fk=\Lie(K)$.
Let $S_\bbR\subset G_\bbR$ be an $\eta_{c,0}$-stable maximal split torus 
contained in a maximal torus $T_\bR\subset G_\bbR$.
Pick a minimal parabolic $P_\bbR\subset G_\bbR$ such that $T_\bbR\subset P_\bbR$
and a Levi factor $L_\bbR$ such that 
$T_\bbR\subset L_\bbR$.
Denote $T_c=T^{\eta_{c,0}}$, $W_c=N_{G_c}(T_c)/T_c\simeq W$.
Note $\theta_0,\eta_0$ have the same action on $W_c$.

We denote by 
$S\subset T\subset L\subset P$
the complexifications.
Then $T$ is stable under all $\theta_0,\eta_0,\eta_{c,0}$.
We pick a Borel subgroup $B$
such that $T\subset B\subset P$.
Denote the Lie algebras by $\fs,\ft,\fb,\fp,\fl.\fg$.

The involution $\theta_0$ (resp. $\eta_0$)
is called quasi-split if 
$B=P$. In this case we have 
$T=B\cap\theta_0(B)$ and $B^-=\theta_0(B)$
is the opposite Borel subgroup.

We denote by $X_*(T)$ the cocharacter lattice and $X_*(T)^+$ the set of dominant coweights 
associated to $B$. 
We have a natural action of $\theta_0$ on $X_*(T)$
given by 
$\theta_0(\lambda)(z)=\theta_0(\lambda(z))$, $z\in\bC^\times=\mathbb G_m(\bC)$. 
Notice that in the quasi-split case $\theta_0$ sends dominant coweights $X_*(T)^+$ to $-X_*(T)^+$.

\subsection{Involutions on loop groups}
We fix $\epsilon\in\{1,-1\}$.
Let $F=\bC(\!(t)\!)$, $\cO=\bC[\![t]\!]$.
Let $G(F)$ and $G(\cO)$ be the formal loop group and formal arc group of $G$.
Consider the twisted involution on $G(F)$:
\begin{equation}
	\theta(\gamma)(t)=\theta_0(\gamma(\epsilon t)).
\end{equation}
We have $\theta(G(\cO))=G(\cO)$.
We denote by $G(F)^\theta$ (resp. $G(\cO)^\theta$)
the $\theta$ fixed points subgroup 
of $G(F)$ (resp. $G(\cO)$).
We denote the $\theta$ anti-fixed points set by
\[
G(F)^{\on{inv}\circ\theta}=\{\gamma\in G(F)|\theta(\gamma)=\gamma^{-1}\}.
\]
Note that $G(F)^{\on{inv}\circ\theta}$ is stable under the $\theta$-conjugation action on $G(F)$ by
\[
h\cdot_\theta \gamma=h\gamma\theta(h)^{-1}.
\]

Let $G[t,t^{-1}]\subset G(F)$
be the polynomial loop group and 
$G[t]\subset G(\cO)$, $G[t^{-1}]\subset G(F)$
be the polynomial arc group and the negative polynomial arc group.
Consider the twisted involution on $G[t,t^{-1}]$:
\begin{equation}
	\eta(\gamma)(t)=\eta_0(\gamma(\epsilon \bar t^{-1})).
\end{equation}
We have $\eta(G[t])=G[t^{-1}]$.
Denote by $G[t,t^{-1}]^\eta$
the $\eta$ fixed point subgroup of $G[t,t^{-1}]$.
Denote  
\[G[t,t^{-1}]^{\on{inv}\circ\eta}=\{\gamma\in G[t,t^{-1}]|\eta(\gamma)=\gamma^{-1}\}.\]
Then $G[t,t^{-1}]^{\on{inv}\circ\eta}$ is stable under 
the $\eta$-conjugation action on $G[t,t^{-1}]$ by 
\[
h\cdot_\eta\gamma=h\gamma\eta(h)^{-1}.
\]

\begin{lem}\label{theta=eta}
For any $\lambda\in X_*(T)$, viewed as element in 
$G[t,t^{-1}]$, we have 
\[\eta(\lambda)=\theta(\lambda).\]
\end{lem}
\begin{proof}
Indeed, 
there is an isomorphism $T\cong\mathbb G_m^n$ such that  
$\eta_{c,0}(\lambda(t))=(\overline{\lambda(t)})^{-1}=\lambda(\bar t^{-1})$ and it follows that 
\[\eta(\lambda)(t)=\eta_0(\lambda(\epsilon \bar t^{-1}))=\theta_0\circ\eta_{c,0}(\lambda(\epsilon \bar t^{-1}))=\theta_0(\lambda(\epsilon t))=\theta(\lambda)(t).\]
\end{proof}

\begin{rem}
Note that $\eta$ does not act on $G(F)$ or $G(\cO)$ 
because it sends $G(\!(t)\!)$ to $G(\!(t^{-1})\!)$.
\end{rem}

\section{Spherical orbits parametrization}\label{spherical orbits para}
In this section we classify spherical orbits on anti-fixed points by twisted conjugation, for $\theta$ and $\eta$ respectively.
The Matsuki correspondence will follow from comparing the two classifications.
We also study orbit intersections and the induced correspondence
on the affine Grassmannians.

\subsection{Complex spherical orbits}\label{complex orbits}
\subsubsection{}
By Cartan decomposition, we have 
\[
G(F)^{\inv\circ\theta}=\bigsqcup_{\lambda\in X_*(T)^+}
(G(\cO)t^\lambda G(\cO))^{\inv\circ\theta}.
\]

Any $g_1t^\lambda g_2\in(G(\cO)t^\lambda G(\cO))^{\inv\circ\theta}$,
$g_1,g_2\in G(\cO)$ satisfies
$g_1t^\lambda g_2=\theta(g_2)^{-1}t^{-\theta_0(\lambda)}\epsilon^{\theta_0(\lambda)}\theta(g_1)^{-1}$,
so that $t^\lambda$ and $t^{-\theta_0(\lambda)}$
are in the same $G(\cO)$-double coset,
$\lambda$ and $-\theta_0(\lambda)$
are in the same Weyl group orbit.
We assume $\theta_0(B^-)=\Ad_{w_1}B$,
where we choose $w_1\in N_{G_c}(T_c)$.
Then we have
\begin{equation}\label{eq:lambda,general}
	\lambda=-w_1^{-1}\theta_0(\lambda).
\end{equation}
Note that when $\theta_0$ is quasi-split, $w_1=1$.

Define parabolic subgroup $P_\lambda$, 
its Levi $L_\lambda$, and its unipotent radical $U_\lambda$
associated to $\lambda$,
with Lie algebras $\fp_\lambda,\fl_\lambda,\fu_\lambda$.
Let $P_\lambda^-$ be the opposite parabolic, 
with unipotent radical $U_\lambda^-$.
Then $\Ad_{w_1^{-1}}\theta_0$ preserves $L_\lambda$,
maps $P_\lambda, U_\lambda$ to $P_\lambda^-,U_\lambda^-$.
Denote $w_2=\theta_0(w_1)w_1$.
Observe $\lambda=-w_1^{-1}\theta_0(-w_1^{-1}\theta_0(\lambda))
=w_2^{-1}\lambda$,
so that $w_2\in L_\lambda$.
For such $\lambda$, we have
\begin{equation}\label{eq:A_lambda general}
	\begin{split}
		(t^\lambda G(\cO))^{\inv\circ\theta}
		&=\{t^\lambda gw_1^{-1}\mid g\in G(\cO),\ 
		\Ad_{t^\lambda} g=w_2\Ad_{w_1^{-1}}\theta(g)^{-1}\epsilon^\lambda\}\\
		&=:\{t^\lambda gw_1^{-1}\mid g\in A_\lambda\subset G(\cO)\cap\Ad_{t^{-\lambda}}G(\cO)\}.
	\end{split}
\end{equation}

\subsubsection{}
Let subspace $A_{\lambda,0}\subset A_\lambda$ be
\begin{equation}\label{eq:A_lambda,0 general}
	A_{\lambda,0}=\{g_0\in L_\lambda\mid g_0=w_2\Ad_{w_1^{-1}}\theta_0(g_0)^{-1}\epsilon^\lambda\}.
\end{equation}

Twisted conjugation of $h\in G(\cO)\cap\Ad_{t^\lambda}G(\cO)$(resp. $h\in L_\lambda$) 
on $A_\lambda$ (resp. $A_{\lambda,0}$) is given by
\[
h\cdot g=(\Ad_{t^{-\lambda}}h)g(\Ad_{w_1^{-1}}\theta(h))^{-1}
\quad(\text{resp. }h\cdot g=hg\Ad_{w_1^{-1}}\theta_0(h^{-1})).
\]

\begin{prop}\label{p:coset conj classes,general}
	The inclusion $A_{\lambda,0}\subset A_\lambda$ induces a bijection
	\[
	L_\lambda\backslash A_{\lambda,0}\simeq 
	(G(\cO)\cap\Ad_{t^\lambda}G(\cO))\backslash A_\lambda.
	\]
\end{prop}
\begin{proof}	
	We inductively show that $g\in A_\lambda$
	can be twisted conjugated into $A_{\lambda,0}$ up to 
	$G_k(\cO):=\ker(G(\cO)\rightarrow G(\cO/t^k))$,
	$\forall k>0$.
	This will imply the surjectivity of the statement 
	by taking a (convergent) product of
	the twisted conjugations.
	
	Write $g=g_0\tilde{g}_1$ where $g_0\in G, \tilde{g}_1\in G_1(\cO)$.
	From \eqref{eq:A_lambda general}, $g\in G(\cO)\cap\Ad_{t^{-\lambda}}G(\cO)$,
	so that $g_0\in G\cap \Ad_{t^{-\lambda}}G(\cO)=P_\lambda$
	and $\tilde{g}_1\in G_1(\cO)\cap\Ad_{t^{-\lambda}}G(\cO)$.
	Write $g_0=g_{0,1}g_{0,2}$ where $g_{0,1}\in L_\lambda,g_{0,2}\in U_\lambda$.
	Observe that $\Ad_{t^\lambda}g_{0,1}=g_{0,1}$,
	$\Ad_{t^\lambda}g_{0,2}\in G_1(\cO)$,
	$\Ad_{t^\lambda}\tilde{g}_1\in U_\lambda^- G_1(\cO)$.
	Therefore, from \eqref{eq:A_lambda general} we obtain that $g_{0,1}\in A_{\lambda,0}$.
	
	Let $h=\Ad_{w_1^{-1}}\theta_0(g_{0,2})\in U_\lambda^-\subset G(\cO)\cap\Ad_{t^\lambda}G(\cO)$.
	Observe that $\Ad_{t^{-\lambda}}U_\lambda^-\subset G_1(\cO)$.
	We get
	\[
	h\cdot g\in g_{0,1}G_1(\cO).
	\]
	This establish the induction basis for $k=1$.
	Moreover, observe that $G(\cO)\cap\Ad_{t^\lambda}G(\cO)\subset P_\lambda^-G_1(\cO)$.
	Thus two elements in $A_{\lambda,0}G_1(\cO)\cap A_\lambda$
	are twisted conjugated by $G(\cO)\cap\Ad_{t^\lambda}G(\cO)$
	if and only if their factors in $A_{\lambda,0}$
	are twisted conjugated by $L_\lambda$.
	This implies the injectivity of the statement.
	
	It remains to establish the induction.
	Assume the induction hypothesis holds for $k>0$.
	We need to show that any $g\in A_{\lambda,0}G_k(\cO)\cap A_\lambda$
	can be twisted conjugated into $A_{\lambda,0}G_{k+1}(\cO)$.
	Write $g=g_0g_k\tilde{g}_{k+1}$
	where $g_0\in A_{\lambda,0}$,
    $g_k=\exp(t^kX),X\in\fg$, 
    and $\tilde{g}_{k+1}\in G_{k+1}(\cO)$.
	Note that $g\in A_\lambda$ and 
	$g_0\in L_\lambda$ commutes with $t^\lambda$. 
	We deduce that
	\[
	g_k\tilde{g}_{k+1}\in G_k(\cO)\cap\Ad_{t^{-\lambda}}G_k(\cO)
	\subset\exp(t^k\fp_\lambda)G_{k_1}(\cO),
	\]
	i.e. $X\in\fp_\lambda$.
	Write $X=X_1+X_2$, $X_1\in\fl_\lambda$, $X_2\in\fu_\lambda$.
	Twisted conjugating by $h=\exp((\epsilon t)^k\Ad_{w_1^{-1}}\theta_0(X_2))$,
	we can eliminate $X_2$ without changing $g_0$ and $X_1$.
	Thus we may assume $X_2=0$.
	The equation \eqref{eq:A_lambda general} implies
	\begin{equation}\label{eq:3}
		\exp(t^kX_1)\tilde{g}_{k+1}=
        \Ad_{g_0^{-1}w_2g_0}[
		(\Ad_{t^{-\lambda}g_0^{-1}w_1^{-1}}\theta_0(\tilde{g}_{k+1}(\epsilon t)^{-1}))
		\exp(-(\epsilon t)^k\Ad_{g_0^{-1}w_1^{-1}}\theta_0(X_1))].	
	\end{equation}
	
	Note that 
	$a=\Ad_{t^{-\lambda}g_0^{-1}w_1^{-1}}\theta_0(\tilde{g}_{k+1}(\epsilon t)^{-1})
	\in\Ad_{t^{-\lambda}}G_{k+1}(\cO)$.
	Also, we can see from \eqref{eq:3} that
	$a\in\exp(t^k\fl_\lambda)G_{k+1}(\cO)$.
	Since $\Ad_{t^{-\lambda}}$ acts trivially on $\fl_\lambda$,
	we obtain 
	\[
	\exp(t^k\fl_\lambda)G_{k+1}(\cO)\cap\Ad_{t^{-\lambda}}G_{k+1}(\cO)
	=G_{k+1}(\cO)\cap\Ad_{t^{-\lambda}}G_{k+1}(\cO).
	\]
	Thus $a\in G_{k+1}(\cO)$.
	We deduce from \eqref{eq:3} that
	\begin{equation}\label{eq:X_4}
		X_1=-\epsilon^k\Ad_{g_0^{-1}\theta_0(w_1)}\theta_0(X_1).
	\end{equation}
	
	Now let $h=\exp(\frac{1}{2}(\epsilon t)^k\theta_0(\Ad_{w_1}X_1))$.
	We obtain $h\cdot g\in g_0G_{k+1}(\cO)$ as desired.
	
	For each $k>0$, the twisted conjugation is by an element
	$h\in\exp(t^k\fg)$.
	Their product converges to an element of $G(\cO)$.
	This completes the proof of the statement.
\end{proof}

\begin{lem}\label{l:pi_0}
	$L_\lambda\backslash A_{\lambda,0}\simeq\pi_0(A_{\lambda,0})$.
\end{lem}
\begin{proof}
    Comparing the equation of $A_{\lambda,0}$ with the twisted action,
we can see that every $L_\lambda$-orbit on $A_{\lambda,0}$
is a submersion, thus open.
Since $A_{\lambda,0}$ is noetherian,
there are only finitely many open orbits.
Thus they are also closed.
The proof is complete.
\end{proof}

Combining the above, we obtain:
\begin{prop}
	For $\lambda=-w_1^{-1}\theta_0(\lambda)$,
	$G(\cO)\backslash_{\theta}(G(\cO)t^\lambda G(\cO))^{\inv\circ\theta}\simeq\pi_0(A_{\lambda,0})$.
\end{prop}

\subsubsection{}
Observe that for $g_0\in L_\lambda$, 
$\Ad_{\lambda(\bR^\times)}\eta_{c,0}(g_0)=\eta_{c,0}(\Ad_{\lambda(\bR^\times)}g_0)=\eta_{c,0}(g_0)$,
so that $\eta_{c,0}$ acts on $L_\lambda$.
Denote $L_{\lambda,c}=L_\lambda^{\eta_{c,0}}$.
Define
\begin{equation}\label{eq:C_lambda}
	C_{\lambda,0}:=L_{\lambda,c}\cap A_{\lambda,0}
	=\{g_0\in L_{\lambda,c}\mid g_0=w_2\Ad_{w_1^{-1}}\theta_0(g_0)^{-1}\epsilon^\lambda\}
\end{equation}
on which $L_{\lambda,c}$ acts.

\begin{prop}\label{p:C_lambda to A_lambda,0}
	The natural map
	$L_{\lambda,c}\backslash C_{\lambda,0}\rightarrow L_\lambda\backslash A_{\lambda,0}$
	is a bijection.
\end{prop}
\begin{proof}
	First we show surjectivity.
	Denote the Lie algebra of $L_\lambda$ by $\fl_\lambda$,
	which decomposes under $\eta_c$ as
	\[
	\fl_\lambda=\fk_\lambda\oplus\fp_{c,\lambda}.
	\]
	Recall $\Ad_{w_1^{-1}}\circ\theta_0$ acts on $L_\lambda$.
	Since it commutes with $\eta_c$, 
	it acts on $\fk_\lambda,\fp_{c,\lambda}$.
	By Cartan decomposition,
	we can write $g_0=k\exp(Y)\in A_{\lambda,0}$ 
	where $k\in L_{\lambda,c}$ and $Y\in\fp_{c,\lambda}$ satisfy
	\[
	k\exp(Y)
	=(w_2\Ad_{w_1^{-1}}\theta_0(k^{-1})\epsilon^\lambda)
    (\Ad_{w_1^{-1}}\theta_0(\Ad_k\exp(-Y))).
	\]
	Thus we obtain
	\[
	k=w_2\Ad_{w_1^{-1}}\theta_0(k^{-1})\epsilon^\lambda,\quad 
    Y=-\Ad_{w_1^{-1}}\theta_0(\Ad_k Y),
	\]
	so that
	\[
	\exp(-\frac{1}{2}\Ad_k Y)\cdot g_0=k.
	\]
	This gives the surjectivity.
	
	Next we show the injectivity.
	Suppose $h\cdot k_1=k_2$ for $k_1,k_2\in C_{\lambda,0}, h\in L_\lambda$.
	Write $h=k\exp(Y)$ its Cartan decomposition, $k\in L_{\lambda,c}$.
	Then we have
   \begin{equation}\label{cartan}
   \begin{split}
    &h\cdot k_1=k\exp(Y)k_1\exp(-\Ad_{w_1^{-1}}\theta_0(Y))\Ad_{w_1^{-1}}\theta_0(k)^{-1}=k_2\\
   \Leftrightarrow\ &
        kk_1\exp(\Ad_{k_1^{-1}}Y)
        =k_2\Ad_{w_1^{-1}}\theta_0(k)\exp(\Ad_{w_1}^{-1}\theta_0(Y))\\
    \Leftrightarrow\ & 
        kk_1\Ad_{w_1^{-1}}\theta_0(k^{-1})=k_2\ \ \text{and}\ \ \exp(\Ad_{k_1^{-1}}Y)=\exp(\Ad_{w_1}^{-1}\theta_0(Y)).
        \end{split}
        \end{equation}
In particular, we see that  $k_1,k_2$ are also in the same $L_{\lambda,c}$-class.
\end{proof}

\quash{
\Define
\begin{equation}
    T_{\lambda,0}=C_{\lambda,0}\cap T_c=\{g_0\in T_c\mid g_0=\theta_0(g_0^{-1})\epsilon^\lambda\}.
\end{equation}

\begin{prop}
    The natural map 
    $T_{\lambda,0}\rightarrow L_{\lambda,c}\backslash C_{\lambda,0}$
    is surjective.
\end{prop}
\begin{proof}
    This is because in the compact Lie group $L_{\lambda,c}$,
    every element can be twisted conjugated into the 
    $\theta_0$-stable maximal torus $T_c$.
    \textcolor{red}{Can we really reduce to outer automorphisms?}
\end{proof}
}

\subsubsection{}
From the above discussion, we obtain:
\begin{thm}\label{t:theta spherical orbits}
The assignment
$x_\lambda=t^\lambda g_0w_1^{-1}\in G(F)^{\on{inv}\circ\theta}\mapsto g_0\in C_{\lambda,0}$
induces a 
natural bijection between $G(\mathcal O)$ and $L_{\lambda,c}$-orbits 
:
	\[
	G(\cO)\backslash_\theta(G(F))^{\inv\circ\theta}
	\simeq\bigsqcup_{\lambda=-w_1^{-1}\theta_0(\lambda)\in X_*(T)^+}
    L_\lambda\backslash A_{\lambda,0}
   \simeq\bigsqcup_{\lambda=-w_1^{-1}\theta_0(\lambda)\in X_*(T)^+}L_{\lambda,c}\backslash C_{\lambda,0}.
	\]
    Moreover, the natural inclusion $Z_{L_{\lambda,c}}(g_0)\to Z_{G(\mathcal O)}(x_\lambda)$
    of stabilizers
    induces an isomorphism on component groups 
$\pi_0(Z_{L_{\lambda,c}}(g_0))\cong\pi_0(Z_{G(\mathcal O)}(x_\lambda))$.
\end{thm}
\begin{proof}
    The bijection follows from Proposition 
    \ref{p:coset conj classes,general} and \ref{p:C_lambda to A_lambda,0}.
    For the second claim, we observe that 
    the evaluation map 
    and the Cartan decomposition of $L_\lambda$ 
    induces   homomorphisms 
    $Z_{G(\mathcal O}(x_\lambda)=Z_{G(\mathcal O)\cap\Ad_{t^\lambda}(G(\mathcal O)}(x_\lambda)\to Z_{L_\lambda}(g_0)$, $\gamma(t)\to \gamma(0)$, 
     and 
     $Z_{L_\lambda}(g_0)\to Z_{L_\lambda,c}(g_0)$, $h=k\on{exp}(Y)\to k$ (see~\eqref{cartan}) with contractible fibers.
     Moreover, the composed map 
     $Z_{G(\mathcal O)}(x_\lambda)\to Z_{L_\lambda}(g_0)\to Z_{L_\lambda,c}(g_0)$ defines a left inverse of the inclusion map 
     $Z_{L_\lambda,c}(g_0)\to Z_{G(\mathcal O)}(x_\lambda)$. 
     The claim follows.
\end{proof}

\begin{rem}\mbox{}
\begin{itemize}
    \item [(i)]
    Note the the proof for the first bijection 
    in the above Theorem \ref{t:theta spherical orbits}
    is purely algebraic,
    thus works also over an algebraically closed base field
    with sufficiently large characteristic.
    \item [(ii)]
    Some sets $C_{\lambda,0}$ in the above parametrization 
    could be empty, see \S\ref{sss:GL2R epsilon=-1}
    for an example.
    We do not yet have a criteria telling 
    when is $C_{\lambda,0}$ nonempty.
\end{itemize}
    
\end{rem}

\subsection{Real spherical orbits}
\subsubsection{}
By Cartan decomposition, we have 
\[
G[t,t^{-1}]^{\inv\circ\eta}=\bigsqcup_{\lambda\in X_*(T)^+}
(G[t]t^\lambda G[t^{-1}])^{\inv\circ\eta}.
\]
Consider $\eta$-twisted conjugacy classes on $G[t,t^{-1}]^{\inv\circ\eta}$.
Since $\eta(G[t])=G[t^{-1}]$,
the classes in the double coset $G[t]t^\lambda G[t^{-1}]$
are the same as classes in $t^\lambda G[t^{-1}]$.
By Lemma \ref{theta=eta}, we have 
\[
\eta(t^\lambda)=\theta(t^\lambda)=
\theta_0((\epsilon t)^\lambda)=\epsilon^{\theta_0(\lambda)}t^{\theta_0(\lambda)}.
\]
Thus for $t^\lambda g\in (t^\lambda G[t^{-1}])^{\inv\circ\eta}$, we have 
\[
g^{-1}t^{-\lambda}=\epsilon^{\theta_0(\lambda)}t^{\theta_0(\lambda)}\eta(g)
\]
where $\eta(g)\in G[t]$. 
We obtain the same condition as \eqref{eq:lambda,general}:
\begin{equation}\label{eq:real lambda}
	\lambda=-w_1^{-1}\theta_0(\lambda).
\end{equation}
Recall we chose $w_1\in G_c$,
so that $w_2=\theta_0(w_1)w_1=\eta_0(w_1)w_1$.

Denote
\begin{equation}\label{eq:B_lambda general}
	\begin{split}
		(t^\lambda G[t^{-1}])^{\inv\circ\eta}
		&=\{t^\lambda gw_1^{-1}\mid g\in G[t^{-1}],\ 
		\Ad_{t^\lambda}g=w_2\Ad_{w_1^{-1}}\eta(g^{-1})\epsilon^\lambda\}\\
		&=:\{t^\lambda gw_1^{-1}\mid g\in B_\lambda\}.
	\end{split}
\end{equation}

\subsubsection{}
Define $P_\lambda$, $L_\lambda$, $U_\lambda$, 
$P_\lambda^-$, $U_\lambda^-$ as before.
By \eqref{eq:real lambda},
$\Ad_{w_1^{-1}}\eta_0$ acts on $L_\lambda$
and preserves $P_\lambda, P_\lambda^-, U_\lambda, U_\lambda^-$.
Note that 
\[
B_\lambda\subset 
G[t^{-1}]\cap\Ad_{t^{-\lambda}}G[t]
=P_\lambda[t^{-1}]\cap\Ad_{t^{-\lambda}}P_\lambda[t]
=(U_\lambda[t^{-1}]\cap\Ad_{t^{-\lambda}}U_\lambda[t])L_\lambda.
\]

Let subspace $B_{\lambda,0}\subset B_\lambda$ be
\begin{equation}\label{eq:B_lambda,0 general}
	B_{\lambda,0}=\{g_0\in L_\lambda\mid g_0=w_2\Ad_{w_1^{-1}}\eta_0(g_0^{-1})\epsilon^\lambda\}.
\end{equation}

Twisted conjugation of $h\in G[t]\cap\Ad_{t^\lambda}G[t^{-1}]$ on $B_\lambda$ 
(resp. $L_\lambda$ on $B_{\lambda,0}$) is given by
\[
h\cdot g=(\Ad_{t^{-\lambda}}h)g\Ad_{w_1^{-1}}\eta(h)^{-1}),
\quad(\text{resp. }h\cdot g=hg\Ad_{w_1^{-1}}\eta_0(h^{-1})).
\]

\begin{prop}\label{p:coset conj classes,general, real}
	The natural map
	$L_\lambda\backslash B_{\lambda,0}\rightarrow 
	(G([t]\cap\Ad_{t^\lambda}G[t^{-1}])\backslash B_\lambda$
	is a bijection.
\end{prop}
\begin{proof}
    Observe $G[t]\cap\Ad_{t^\lambda}G[t^{-1}]=P_\lambda[t]\cap\Ad_{t^\lambda}P_\lambda[t^{-1}]$.
    The injectivity is then easy.
    
    For surjectivity,
    let $g=u\ell\in B_\lambda$ where
    $\ell\in L_\lambda, u\in U_\lambda[t^{-1}]\cap\Ad_{t^{-\lambda}}U_\lambda[t]$.
    They satisfy
    \[
    \ell\in B_{\lambda,0},\quad
    \Ad_{t^\lambda}u=\Ad_{\ell\epsilon^\lambda w_1^{-1}}\eta(u^{-1}).
    \]
    Since $U_\lambda[t^{-1}]\cap\Ad_{t^{-\lambda}}U_\lambda[t]$ is unipotent
    on which $\Ad_{t^{-\lambda}\ell\epsilon^\lambda}\Ad_{w_1^{-1}}\eta$
    acts as group automorphism,
    $u$ has square root 
    $u^{1/2}\in U_\lambda[t^{-1}]\cap\Ad_{t^{-\lambda}}U_\lambda[t]$
    such that 
    $\Ad_{t^\lambda}u^{1/2}=\Ad_{\ell\epsilon^\lambda w_1^{-1}}\eta(u^{-1/2})$.
    Let $h=\Ad_{t^{-\lambda}}u^{-1/2}$.
    Then 
    \[
    h\cdot g=\ell\in B_{\lambda,0}.
    \]
\end{proof}

Observe that $\eta_0$ coincides with $\theta_0$ on $L_{\lambda,c}$.
Thus
\[
C_{\lambda,0}=L_{\lambda,c}\cap B_{\lambda,0}.
\]

\begin{prop}\label{p:C_lambda to B_lambda,0}
	The natural map
	$L_{\lambda,c}\backslash C_{\lambda,0}\rightarrow L_\lambda\backslash B_{\lambda,0}$
	is a bijection.
\end{prop}
\begin{proof}
	Simply replace $\theta_0$ with $\eta_0$ in the proof of
	Proposition \ref{p:C_lambda to A_lambda,0}.
\end{proof}

\subsubsection{}
Combining the above discussion, we obtain:
\begin{thm}\label{t:eta spherical orbits}
	We have natural bijections induced by 
    $x_\lambda =t^\lambda g_0w_1^{-1}\mapsto g_0\in C_{\lambda,0}$:
    \[    
    G[t]\backslash_\eta(G[t,t^{-1}])^{\inv\circ\eta}
    \simeq\bigsqcup_{\lambda=-w_1^{-1}\theta_0(\lambda)\in 
    X_*(T)^+}L_\lambda\backslash B_{\lambda,0}
	\simeq\bigsqcup_{\lambda=-w_1^{-1}\theta_0(\lambda)\in 
    X_*(T)^+}L_{\lambda,c}\backslash C_{\lambda,0}.
	\]
    Moreover, the natural inclusion $Z_{L_{\lambda,c}}(g_0)\to Z_{G[t]}(x_\lambda)$
    of stabilizers
    induces an isomorphism on component groups 
$\pi_0(Z_{L_{\lambda,c}}(g_0))\cong\pi_0(Z_{G[t]}(x_\lambda))$.
\end{thm}
\begin{proof}
    Same proof as Theorem \ref{t:theta spherical orbits}.
\end{proof}

\subsection{Matsuki duality for spherical orbits}\label{bijection for spherical orbits}
\subsubsection{Intersection of spherical orbits}
Denote the $\theta$-twisted(resp. $\eta$-twisted) 
conjugation by $g\cdot_\theta h$(resp. $g\cdot_\eta h$).
We have seen from Theorem \ref{t:theta spherical orbits}
and Theorem \ref{t:eta spherical orbits}
that both 
$\theta$-twisted $G(\cO)$-orbits on $(G(F))^{\inv\circ\theta}$
and
$\eta$-twisted $G[t]$-orbits on $(G[t,t^{-1}])^{\inv\circ\eta}$
are represented by $t^\lambda C_{\lambda,0}w_1^{-1}$,
where $\theta_0(\lambda)=-w_1\lambda$
and $C_{\lambda,c}\subset L_{\lambda,c}$.
Given $t^\lambda g_0w_1^{-1}\in t^\lambda C_{\lambda,0}w_1^{-1}$,
we now describe the intersection of 
$G(\cO)\cdot_\theta t^\lambda g_0w_1^{-1}$
with $G[t]\cdot_\eta t^\lambda g_0w_1^{-1}$.

\begin{lem}\label{l:intersection 1}
    For any $\lambda\in X_*(T)^+$, we have
    \[
    G[t^{-1}]\cap(\Ad_{t^{-\lambda}}G(\cO))G(\cO)
    =(L^{<0}G\cap\Ad_{t^{-\lambda}}G(\cO))G
    =(L^{<0}G\cap\Ad_{t^{-\lambda}}G[t])G.
    \]
\end{lem}
\begin{proof}
    The second equality is obvious 
    since the intersection is generated by finitely many 
    affine root subgroups.
    So we only need to prove the first equality.
    
    Note that $G$ acts freely on 
    $G[t^{-1}]\cap((\Ad_{t^{-\lambda}}G(\cO))G(\cO)$
    on the right.
    We have
    \begin{align*}
    G[t^{-1}]\cap(\Ad_{t^{-\lambda}}G(\cO)))G(\cO))
    =&(L^{<0}G\cap(\Ad_{t^{-\lambda}}G(\cO))G(\cO))G\\
    =&t^{-\lambda}(t^\lambda L^{<0}G\cap G(\cO)t^\lambda G(\cO))G.
    \end{align*}
    
    We can decompose
    \[
    t^\lambda L^{<0}G
    =(t^\lambda L^{<0}G\cap G(\cO)t^\lambda)
    (L^{<0}G\cap\Ad_{t^{-\lambda}}L^{<0}G).
    \]

    Given any $g=ab\in t^\lambda L^{<0}G\cap G(\cO)t^\lambda G(\cO)$
    where $a\in t^\lambda L^{<0}G\cap G(\cO)t^\lambda$
    and $b\in L^{<0}G\cap\Ad_{t^{-\lambda}}L^{<0}G$.
    Since $ab\in G(\cO)t^\lambda G(\cO)$
    and $a\in G(\cO)t^\lambda$,
    we have $b\in t^{-\lambda}G(\cO)t^\lambda G(\cO)$.
    Observe that the affine Grassmannian slice
    \[
    t^\lambda(L^{<0}G\cap\Ad_{t^{-\lambda}}L^{<0}G)\cap G(\cO)t^{\lambda}G(\cO)=t^\lambda.
    \]
    We obtain
    \[
    b\in t^{-\lambda}(t^\lambda(L^{<0}G\cap\Ad_{t^{-\lambda}}L^{<0}G)\cap G(\cO)t^{\lambda}G(\cO))=1.
    \]
    This implies 
    $t^\lambda L^{<0}G\cap G(\cO)t^\lambda G(\cO)=t^\lambda L^{<0}G\cap G(\cO)t^\lambda$,
    from which the lemma follows.
\end{proof}

\begin{lem}\label{l:intersection 2}
    For any $t^\lambda g_0$ 
    where $\lambda=-w_1^{-1}\theta_0(\lambda)\in X_*(T)^+$
    and $g_0\in C_{\lambda,0}$,
    \[
    G(\cO)\cdot_\theta t^\lambda g_0w_1^{-1}
    \cap G[t]\cdot_\eta t^\lambda g_0w_1^{-1}
    \subset Gt^\lambda G.
    \]
\end{lem}
\begin{proof}
    Assume for $g_1\in G(\cO), g_2\in G[t]$, we have
    \[
    g:=g_1 t^\lambda g_0w_1^{-1}\theta(g_1)^{-1}
     =g_2 t^\lambda g_0w_1^{-1}\eta(g_2)^{-1}.
    \]
    Observe that $g_0\in L_\lambda$ commutes with $t^\lambda$.
    By Lemma \ref{l:intersection 1}, we obtain
    \begin{align*}
    \Ad_{w_1^{-1}}\eta(g_2)^{-1}
    &=t^{-\lambda}g_0^{-1}g_2^{-1}g_1t^\lambda g_0w_1^{-1}\theta(g_1)^{-1}w_1\\
    &\in G[t^{-1}]\cap(\Ad_{t^{-\lambda}}G(\cO))G(\cO)
    =(L^{<0}G\cap\Ad_{t^{-\lambda}}G[t])G.
    \end{align*}
    
    Denote $G[t]_1=\ker(G[t]\rightarrow G)$.
    Observe $\eta(L^{<0}G)=G[t]_1$.
    We get
    $g_2\in G(G[t]_1\cap\Ad_{t^\lambda}G[t^{-1}])$.
    Write $g_2=h_0h_1$, 
    $h_0\in G$, $h_1\in G[t]_1\cap\Ad_{t^\lambda}G[t^{-1}]$.
    Note that the intersection of orbits belongs to
    $(G(F))^{\inv\circ\theta}\cap(G[t,t^{-1}])^{\inv\circ\eta}\subset G[t,t^{-1}]^{\eta_c}$.
    Thus $\eta_c(g)=g$.
    Note that $t^\lambda, w_1, g_0$ are fixed by $\eta_c$
    and $\eta_c\circ\eta=\theta$.
    We obtain
    \[
    g=h_0h_1t^\lambda g_0w_1^{-1}\eta(h_1)^{-1}\eta(h_0)^{-1}
    =\eta_c(h_0)\eta_c(h_1)t^\lambda g_0w_1^{-1}\theta(h_1)^{-1}\theta(h_0)^{-1}.
    \]
    
    Rewrite this as
    \[
    h_1\Ad_{t^\lambda g_0w_1^{-1}}\eta(h_1)^{-1}
    =h_0^{-1}\eta_c(h_0)\eta_c(h_1)(\Ad_{t^\lambda g_0w_1^{-1}}\theta(h_1)^{-1})
    (\Ad_{t^\lambda g_0w_1^{-1}}(\theta(h_0)^{-1}\eta(h_0))).
    \]
    Observe $h_1\in U_\lambda[t]_1$,
    $\Ad_{t^\lambda g_0w_1^{-1}}\eta(h_1)^{-1}\in U_\lambda[t]$,
    $\eta_c(h_1)\in U_\lambda^-[t^{-1}]$,
    $\Ad_{t^\lambda g_0w_1^{-1}}\theta(h_1)^{-1}\in U_\lambda^-[t^{-1}]$.
    We obtain
    \[
    h_1\Ad_{t^\lambda g_0w_1^{-1}}\eta(h_1)^{-1}
    \in U_\lambda[t]\cap G[t^{-1}]\Ad_{t^\lambda}G.
    \]
    Write
    \[
    \Ad_{t^\lambda}G
    =\Ad_{t^\lambda}(\bigcup_{w\in W}U^-w L_\lambda U_\lambda)
    \subset\bigcup_{w\in W}G[t^{-1}]t^{\lambda-w\lambda}G(\Ad_{t^\lambda}U_\lambda).
    \]
    Thus
    \[
    U_\lambda[t]\cap G[t^{-1}]\Ad_{t^\lambda}G
    \subset\bigcup_{w\in W}U_\lambda[t]\cap G[t^{-1}]t^{\lambda-w\lambda}G(\Ad_{t^\lambda}U_\lambda).
    \]
    Note that
    $U_\lambda[t]\cap G[t^{-1}]t^{\lambda-w\lambda}G(\Ad_{t^\lambda}U_\lambda)
    \subset G[t]\cap G[t^{-1}]t^{\lambda-w\lambda}G[t]$
    is nonempty only if $\lambda=w\lambda$.
    Thus
    \[
    h_1\Ad_{t^\lambda g_0w_1^{-1}}\eta(h_1)^{-1}\in
    U_\lambda[t]\cap G[t^{-1}]\Ad_{t^\lambda}U_\lambda
    =U_\lambda\Ad_{t^\lambda}U_\lambda,
    \]
    and
    \[
    g=h_0(h_1\Ad_{t^\lambda g_0w_1^{-1}}\eta(h_1)^{-1})t^\lambda g_0w_1^{-1}\eta(h_0)^{-1}
    \in G U_\lambda t^\lambda U_\lambda t^{-\lambda}t^\lambda G
    =G t^\lambda G.
    \]
\end{proof}

\begin{prop}\label{p:spherical orbits intersection}
    For any $t^\lambda g_0w_1^{-1}$ 
    where $\lambda=-w_1^{-1}\theta_0(\lambda)\in X_*(T)^+$
    and $g_0\in C_{\lambda,0}$,
    \[
    G(\cO)\cdot_\theta t^\lambda g_0w_1^{-1}
    \cap G[t]\cdot_\eta t^\lambda g_0w_1^{-1}
    =G_c\cdot_\theta t^\lambda g_0w_1^{-1}
    =G_c\cdot_\eta t^\lambda g_0w_1^{-1}.
    \]
\end{prop}
\begin{proof}
    The second equality is because 
    $\theta$ and $\eta$ coincide on $G_c$.

    To show the first equality,
    we first show that
    \[
    A:=G(\cO)\cdot_\theta t^\lambda g_0w_1^{-1}
    \cap G[t]\cdot_\eta t^\lambda g_0w_1^{-1}
    \subset G\cdot_\theta t^\lambda g_0w_1^{-1}.
    \]
    By Lemma \ref{l:intersection 2},
    \[
    A\subset G(\cO)\cdot_\theta t^\lambda g_0w_1^{-1}\cap Gt^\lambda G
    \subset G(\cO)\cdot_\theta t^\lambda g_0w_1^{-1}
    \cap(G t^\lambda G)^{\inv\circ\theta}.
    \]
    The proof of Theorem \ref{t:theta spherical orbits}
    shows that the $\theta$-twisted $G$-orbits 
    in $(G t^\lambda G)^{\inv\circ\theta}$
    are also parametrized by $L_{\lambda,c}\backslash C_{\lambda,0}$.
    Thus 
    $G(\cO)\cdot_\theta t^\lambda g_0w_1^{-1}
    \cap(G t^\lambda G)^{\inv\circ\theta}
    =G\cdot_\theta t^\lambda g_0w_1^{-1}$.

    Recall $A\subset (G[t,t^{-1}])^{\eta_c}$. 
    We obtain from the above that
    \[
    A\subset(G\cdot_\theta t^\lambda g_0w_1^{-1})^{\eta_c}.
    \]
    Let $g=k\exp(Y)\in G$, $k\in G_c$ 
    and $Y\in\fp_c$,
    such that $g\cdot_\theta t^\lambda g_0w_1^{-1}$ is fixed by $\eta_c$.
    Recall $\theta$ acts on $G_c$ and $\fp_c$.
    We have
    \begin{align*}
    &k\exp(Y)t^\lambda g_0w_1^{-1}\exp(-\theta(Y))\theta(k)^{-1}
    =k\exp(-Y)t^\lambda g_0w_1^{-1}\exp(\theta(Y))\theta(k)^{-1}\\
    \Rightarrow &
    \exp(2Y)=\Ad_{t^\lambda}\exp(2\Ad_{g_0w_1^{-1}}\theta(Y))\in G\cap\Ad_{t^\lambda}G=L_\lambda\ \text{commutes with }t^\lambda\\
    \Rightarrow &
    Y=\Ad_{g_0w_1^{-1}}\theta(Y)\in\Lie(L_\lambda)\cap\fp_c\\
    \Rightarrow &
    g\cdot_\theta t^\lambda g_0w_1^{-1}=k\cdot_\theta t^\lambda g_0w_1^{-1}.
    \end{align*}
    We conclude that
    \[
    G_c\cdot_\theta t^\lambda g_0w_1^{-1}
    \subset A 
    \subset (G\cdot_\theta t^\lambda g_0w_1^{-1})^{\eta_c}
    =G_c\cdot_\theta t^\lambda g_0w_1^{-1}
    \]
    is an equality.
\end{proof}

\begin{cor}\label{c:union of spherical orbits intersection}
    Denote $t^{X_*(T)}=\{t^\lambda\mid\lambda\in X_*(T)\}$.
    Then
    \[
    \bigsqcup_{\substack{\lambda=-w_1^{-1}\theta_0(\lambda)\in X_*(T)^+
    \\
    \bar{g}_0\in L_{\lambda,c}\backslash C_{\lambda,0}}}
    G(\cO)\cdot_\theta t^\lambda g_0w_1^{-1}
    \cap G[t]\cdot_\eta t^\lambda g_0w_1^{-1}
    =((Gt^{X_*(T)}G)^{\inv\circ\theta})^{\eta_c}.
    \]
    In the above $g_0$ is any representative of the orbit $\bar{g}_0$.
\end{cor}
\begin{proof}
    Observe that
    \[
    (Gt^{X_*(T)}G)^{\inv\circ\theta}
    =(Gt^{X_*(T)^+}G)^{\inv\circ\theta}
    =\bigsqcup_{\lambda=-w_1^{-1}\theta_0(\lambda)\in X_*(T)^+}
    (Gt^\lambda G)^{\inv\circ\theta}.
    \]
    
    Since the $\theta$-twisted $G$-orbits on $Gt^\lambda G$
    are represented by $L_{\lambda,c}$-orbits on $C_{\lambda,0}$,
    we have
    \[
    (Gt^\lambda G)^{\inv\circ\theta}=G\cdot_\theta t^\lambda C_{\lambda,0}w_1^{-1}.
    \]
    
    Recall in the proof of 
    Proposition \ref{p:spherical orbits intersection}
    we show that
    \[
    (G\cdot_\theta t^\lambda g_0w_1^{-1})^{\eta_c}
    =G_c\cdot_\theta t^\lambda g_0w_1^{-1}.
    \]
    To sum up, we deduce
    \[
    ((Gt^{X_*(T)}G)^{\inv\circ\theta})^{\eta_c}
    =\bigsqcup_{\substack{\lambda=-w_1^{-1}\theta_0(\lambda)\in X_*(T)^+
    \\
    \bar{g}_0\in L_{\lambda,c}\backslash C_{\lambda,0}}}
    G_c\cdot_\theta t^\lambda g_0w_1^{-1}.
    \]
    The desired statement follows from 
    Proposition \ref{p:spherical orbits intersection}.
\end{proof}

\subsubsection{Matsuki duality for spherical orbits}

\begin{thm}\label{t:affine Matsuki}
There is a unique bijection
	\[
	G(\cO)\backslash_\theta(G(F))^{\inv\circ\theta}
	\simeq
	G[t]\backslash_\eta(G[t,t^{-1}])^{\inv\circ\eta}
	\]
    such that a $\theta$-tiwsted $G(\cO)$-orbit  $\mathcal O_X$
    corresponds to an $\eta$-twisted $G[t]$-orbit $\mathcal O_\bR$
    if and only if their intersection 
$\mathcal O_X\cap\mathcal O_\bR$
is a single twisted $G_c$-orbit 
    in
    $((Gt^{X_*(T)}G)^{\inv\circ\theta})^{\eta_c}$.
    Moreover, for any $x\in \mathcal O_X\cap\mathcal O_\bR$, there is a natural isomorphism between component groups of stabilizers 
    \[\pi_0(Z_{G(\mathcal O)}(x))\cong \pi_0(Z_{G[t]}(x)).\]
    
\end{thm}
\begin{proof}
Combining 
 Theorem \ref{t:theta spherical orbits}
 and 
Theorem \ref{t:eta spherical orbits}
we obtain the bijection 
\begin{equation}\label{bijection}
\begin{split}
G(\cO)\backslash_\theta(G(F))^{\inv\circ\theta}
&\simeq\bigsqcup_{\lambda=-w_1^{-1}\theta_0(\lambda)\in X_*(T)^+}
L_\lambda\backslash A_{\lambda,0}
\simeq \bigsqcup_{\lambda=-w_1^{-1}\theta_0(\lambda)\in X_*(T)^+}L_{\lambda,c}\backslash C_{\lambda,0}\\
&\simeq\bigsqcup_{\lambda=-w_1^{-1}\theta_0(\lambda)\in 
X_*(T)^+}L_{\lambda,c}\backslash B_{\lambda,0}\simeq G[t]\backslash_\eta(G[t,t^{-1}])^{\inv\circ\eta}
\end{split}
\end{equation}
and the claim about isomorphism of  component groups of centralizers.
The claim for intersection of orbits 
follows from Corollary \ref{c:union of spherical orbits intersection}.
\end{proof}

\begin{cor}\label{c:finite G-Matsuki}

There is a bijection. 
\[L_\lambda\backslash A_{\lambda,0}\simeq L_{\lambda,c}\backslash C_{\lambda,c}\simeq L_\lambda\backslash B_{\lambda,0}\]
for  each $\lambda=-w_1^{-1}\theta_0(\lambda)\in X_*(T)^+$.
In particular, when $\lambda=0$, 
 we obtain a bijection
\begin{equation}\label{cartan cohomology}
	G\backslash_{\theta_0}G^{\inv\circ\theta_0}
	\simeq
	G\backslash_{\eta_0}G^{\inv\circ\eta_0}.
	\end{equation}
characterized by the property that the intersection of the corresponding $G$-orbits
    is contained in $G_c^{\on{inv}\circ\theta_0}=G_c^{\on{inv}\circ\eta_0}$,
    in which case it is a $G_c$-orbit.   
\end{cor}

\begin{rem}\label{AT}
    Note that  the set $G\backslash_{\theta_0}G^{\inv\circ\theta_0}$ (resp. $G\backslash_{\eta_0}G^{\inv\circ\eta_0}$)
is the Cartan cohomology (resp. Galois cohomology) for the real group $G_\bR$, denote by 
$\mathrm H^1(\theta_0,G)$ (resp. $\mathrm H^1(\eta_0,G)$) in  \cite{AT}, 
and the 
 bijection~\eqref{cartan cohomology}
is one of the main results in \cite[Theorem 1.1]{AT}.
\end{rem}
    
\subsection{Matsuki duality for affine Grassmannians}
For $\epsilon=1$,
it is proved in \cite[Theorem 1.2]{N} that 
there is a bijection between
$G(F)^\theta$-orbits and $G[t^{\pm1}]^\eta$-orbits
on $\Gr=G(\cO)\backslash G(F)$
such that then intersection of corresponding orbits
is a single $LK_c=(G(F)^\theta)^\eta$-orbit.
We generalize this to $\epsilon=\pm1$
and also compare this with the intersection of $G(\cO)$-orbits
on $G(F)^{\inv\circ\theta}$ and $G[t^{\pm1}]^{\inv\circ\eta}$.
In \ref{ss:Iwahori double cosets},
we will establish the Iwahori version
which has not been discussed before even for $\epsilon=1$.

To this end, 
we study the images
$\tau_\theta:G(F)\to G(F)^{\on{inv}\circ\theta}$ and 
$\tau_\eta:G(F)\to G(F)^{\on{inv}\circ\eta}$.
In particular, we show that when $\epsilon=-1$ both maps are surjective.

\subsubsection{Image of $\tau_\theta$}
It is known that for $\epsilon=1$, 
$\tau_\theta$ may not be surjective in general,
for example when $(G,K)=(\GL_n,\mathrm{O}_n)$.
However, when $\epsilon=-1$ we have the following result.

\begin{prop}\label{p:tau_theta surj epsilon=-1}
    The map $\tau_\theta:G(F)\rightarrow G(F)^{\inv\circ\theta}$
    is surjective for $\epsilon=-1$,
    which induces an isomorphism
    $G(F)/G(F)^\theta\simeq G(F)^{\inv\circ\theta}$.
\end{prop}
\begin{proof}
The $G(F)$-orbits 
on $G(F)^{\on{inv}\circ\theta}$
are in bijection with the group cohomology
$H^1(\mu_2,G(F))$ where the generator $-1$ acts by $\theta$.
Note that there is an identification
$H^1(\mu_2,G(F))\cong H^1(\on{Gal}(F/F'),H_{F'}(F))$
where $F'=k((t^2))$ and 
the right hand side is the Galois cohomology of a 
$F/F'$-form $H_{F'}$ of 
$G_{F'}$ such that $H_{F'}(F')=G_\theta(F)$.
 Since the field $F'=k((t^2))$ is of dimension $\leq1$, we have 
 $H^1(\on{Gal}(F'),H_{F'}(F_s))=0$ (where $F_s$ is a separable closure of $F'$ ).
 The desired claim follows from the injection 
$0\to H^1(\on{Gal}(F/F'),H_{F'}(F))\to H^1(\on{Gal}(F'),H_{F'}(F_s))$.
\end{proof}

\subsubsection{Image of $\tau_\eta$}
Recall the based compact loop group
$\Omega G_c=\{\gamma(t)\in G[t^{\pm1}]^{\eta_c}\mid \gamma(1)=1\}$.

It is well known that the inclusion induces natural homeomorphism
$\Omega G_c\simeq\Gr$,
so that $G(F)=\Omega G_c\cdot G(\cO)$ is a unique factorization.
Clearly $G[t^{\pm1}]^{\eta_c}=\Omega G_c\cdot G_c$.
Denote $G(\cO)_\fp=\exp(\fp)G(\cO)_1$.
By Cartan decomposition $G=G_c\cdot\exp(\fp)$,
we have decomposition of space
\begin{equation}\label{eq:Iwasawa decomp loop}
    G(F)=G[t^{\pm1}]^{\eta_c}\cdot G(\cO)_\fp,
\end{equation}
i.e. every element of $g\in G(F)$ 
can be written in a unique way as a product
$kp$, $k\in G[t^{\pm1}]^{\eta_c}$, $p\in G(\cO)_\fp$.
Note that both $G[t^{\pm1}]^{\eta_c}$
and $G(\cO)_\fp$ are $\theta$-stable.

\begin{lem}\label{l:tau_theta cpt preimage}
    We have 
    $\tau_\theta(G[t^{\pm1}]^{\eta_c})=\tau_\theta(G(F))\cap G[t^{\pm1}]^{\eta_c}$.
\end{lem}
\begin{proof}
    The left hand side is obviously contained in the right hand side.
    It suffices to show the converse inclusion.
    Let $g\in G(F)$ be such that 
    $\tau_\theta(g)=g\theta(g)^{-1}\in G[t^{\pm1}]^{\eta_c}$.
    Write $g=kp$, $k\in G[t^{\pm1}]^{\eta_c}$, $p\in G(\cO)_\fp$.
    Then
    \[
    kp\theta(p)^{-1}\theta(k)^{-1}=\tau_\theta(g),\quad
    kp=\tau_\theta(g)\theta(k)\theta(p).
    \]
    
    By uniqueness of the decomposition \eqref{eq:Iwasawa decomp loop},
    we obtain
    \[
    k=\tau_\theta(g)\theta(k),\qquad
    \tau_\theta(g)=k\theta(k)^{-1}=\tau_\theta(k),
    \]
    which completes the proof.
\end{proof}

\quash{
\textcolor{red}{Can we have the above lemma for $\eta$? The lemma for $\theta$ only gives an embedding from $G(\cO)\backslash G(F)/G(F)^\theta$ into $G[t]\backslash G[t^{\pm1}]/G[t^{\pm1}]^\eta$,
but the converse direction requires the lemma for $\eta$.
The problem is that $G(\cO)_\fp$ is not $\eta$-stable.
With both directions, we can lift the Matsuki correspondence
from anti-fixed points to loop group double cosets.}\\
}

\begin{cor}\label{c:cpt image epsilon=-1}
    For $\epsilon=-1$, we have
    \begin{itemize}
        \item [(i)]
        $\tau_\theta(G[t^{\pm1}]^{\eta_c})
        =(G[t^{\pm1}]^{\eta_c})^{\inv\circ\theta}$,

        \item [(ii)]
        $\tau_\eta(G[t^{\pm1}]^{\eta_c})
        =\tau_\eta(G[t^{\pm1}])\cap G[t^{\pm1}]^{\eta_c}$.
    \end{itemize}
\end{cor}
\begin{proof}
    (i):
    This follows immediately from 
    Proposition \ref{p:tau_theta surj epsilon=-1}
    and Lemma \ref{l:tau_theta cpt preimage}.

    (ii):
    We only need to show surjectivity.
    This follows from (i):
    \[
    \tau_\eta(G[t^{\pm1}])\cap G[t^{\pm1}]^{\eta_c}
    \subset (G[t^{\pm1}]^{\inv\circ\eta})^{\eta_c}
    =(G[t^{\pm1}]^{\eta_c})^{\inv\circ\theta}
    =\tau_\theta(G[t^{\pm1}]^{\eta_c})
    =\tau_\eta(G[t^{\pm1}]^{\eta_c}).
    \]
\end{proof}

\begin{prop}\label{p:tau_eta surj epsilon=-1}
    The map $\tau_\eta:G[t^{\pm1}]\rightarrow G[t^{\pm1}]^{\inv\circ\eta}$
    is surjective for $\epsilon=-1$,
    which induces an isomorphism
    $G[t^{\pm1}]/G[t^{\pm1}]^\eta\simeq G[t^{\pm1}]^{\inv\circ\eta}$.
\end{prop}
\begin{proof}
    Since $\tau_\eta$ is equivariant with respect to $G[t]$-action, 
    it suffices to show
    each $G[t]$-orbit on $G[t^{\pm1}]^{\inv\circ\eta}$
    intersects with the image of $\tau_\eta$.
    By Theorem \ref{t:eta spherical orbits},
    any orbit contains an element fixed by $\eta_c$.
    By Corollary \ref{c:cpt image epsilon=-1}.(i),
    such an element is contained in
    \[
    (G[t^{\pm1}]^{\inv\circ\eta})^{\eta_c}
    =(G[t^{\pm1}]^{\eta_c})^{\inv\circ\theta}
    =\tau_\theta(G[t^{\pm1}]^{\eta_c})
    =\tau_\eta(G[t^{\pm1}]^{\eta_c}),
    \]
    thus is an image of $\tau_\eta$.
    This completes the proof.
\end{proof}

\begin{rem}
The proposition is not true if $\epsilon=1$.
Indeed, consider the case when $G=\mathbb G_m$ and $\eta_0(t)=\bar t^{-1}$
is the compact involution.
Then we have $ G[t,t^{-1}]^{\on{inv}\circ\eta}=\bbR^\times$
and $\tau_{\eta}( G[t,t^{-1}])=\mathbb R_{\geq0}$.

%Indeed, in this case, the image of 
%$G[t]\backslash G[t,t^{-1}]/G[t,t^{-1}]^\eta\to 
%G[t]\backslash G[t,t^{-1}]^{\on{inv}\circ\eta}\cong
%\Bun_G(\mathbb P^1)_\eta(\mathbb R)$ is the union of components 
%consisting of $G_\bbR$-bundles on the real projective line $%\mathbb P^1R$ that admit complex uniformization \cite[Section 6.2]{CN} 

\end{rem}

\subsubsection{Intersection of cosets}
Recall from Theorem \ref{t:affine Matsuki} 
and Corollary \ref{c:union of spherical orbits intersection}
that $\theta$-twisted $G(\cO)$-orbits are parametrized by
twisted $G_c$-orbits on $((Gt^{X_*(T)}G)^{\inv\circ\theta})^{\eta_c}$.
By Lemma \ref{l:tau_theta cpt preimage},
any element in $\tau_\theta(G(F))\in((Gt^{X_*(T)}G)^{\inv\circ\theta})^{\eta_c}$
has a preimage in $G[t^{\pm1}]^{\eta_c}$ under $\tau_\theta$.

Henceforth, we denote 
$\tau=\tau_\theta|_{G(F)^{\eta_c}}=\tau_\eta|_{G(F)^{\eta_c}}$.

\begin{prop}\label{p:spherical double cosets intersection}
    For any $x\in G[t^{\pm1}]^{\eta_c}$ such that
    $\tau_\theta(x)\in((Gt^{X_*(T)}G)^{\inv\circ\theta})^{\eta_c}$,
    we have
    \begin{itemize}
        \item [(i)]
        $
        G(\cO)x G(F)^\theta\cap G[t]x G[t^{\pm1}]^\eta\cap G[t^{\pm1}]^{\eta_c}
        =G_c x G[t^{\pm1}]^{\theta,\eta}
        $,
        
        \item [(ii)]
        $
        G(\cO)x G(F)^\theta\cap G(\cO)x G[t^{\pm1}]^\eta
        =G(\cO)x G[t^{\pm1}]^{\theta,\eta}
        $.
    \end{itemize}
\end{prop}
\begin{proof}
    (i):
    Observe that $\tau$ induces an embedding
    \begin{align*}
    \tau:G_c\backslash(G(\cO)xG(F)^\theta\cap G[t]xG[t^{\pm1}]^\eta\cap G[t^{\pm1}]^{\eta_c})/G[t^{\pm1}]^{\eta_c,\theta}
    &\hookrightarrow
    G_c\backslash G[t^{\pm1}]^{\eta_c}/G[t^{\pm1}]^{\eta_c,\theta}\\
    &\hookrightarrow
    G_c\backslash_\theta (G[t^{\pm1}]^{\eta_c})^{\inv\circ\theta}.
    \end{align*}

    By Theorem \ref{t:affine Matsuki}, we have
    \begin{align*}
    \tau(G(\cO)xG[t^{\pm1}]^\theta\cap G[t]xG[t^{\pm1}]^\eta\cap G[t^{\pm1}]^{\eta_c})
    &\subset\tau_\theta(G(\cO)xG[t^{\pm1}]^\theta)\cap\tau_\eta(G[t]xG[t^{\pm1}]^\eta)\\
    &=(G(\cO)\cdot_\theta\tau(x))\cap (G[t]\cdot_\eta\tau(x))
    =G_c\cdot_\theta\tau(x).
    \end{align*}

    Thus $G(\cO)xG[t^{\pm1}]^\theta\cap G[t]xG[t^{\pm1}]^\eta\cap G[t^{\pm1}]^{\eta_c}$
    is a single $G_c-G[t^{\pm1}]^{\eta_c,\theta}$ orbit,
    which must be $G_c xG[t^{\pm1}]^{\eta_c,\theta}$.

    (ii):
    We can rewrite decomposition \eqref{eq:Iwasawa decomp loop}
    as $G(F)=G(\cO)G[t^{\pm1}]^{\eta_c}$,
    where the product may not be unique.
    Therefore, any subspace $X\subset G(F)$
    that is stable under left $G(\cO)$-multiplication
    satisfies $X=G(\cO)X^{\eta_c}$.
    In particular, by (i) we have
    \begin{align*} 
    G(\cO)x G(F)^\theta\cap G(\cO)x G[t^{\pm1}]^\eta
    &=G(\cO)(G(\cO)x G(F)^\theta\cap G(\cO)x G[t^{\pm1}]^\eta)^{\eta_c}\\
    &=G(\cO)x G[t^{\pm1}]^{\theta,\eta}.
    \end{align*}
\end{proof}

\subsubsection{Affine Grassmannians}
In the following,
we establish
the  Matsuki duality for affine Grassmannians using Theorem \ref{t:affine Matsuki}, 
extending the case of $\epsilon=1$ studied in
\cite[Theorem 1.2]{N}.

\begin{prop}\label{p:spherical orbits diagram}
    We have a commutative diagram of sets of $G(\cO)$ or $G[t]$ orbits
    \[\xymatrix{
        G(\cO)\backslash G(F)/G(F)^\theta\ar[d]^{\tau_\theta}\ar[r]^{\iota_0}& 
        G[t]\backslash G[t^{\pm1}]/G[t^{\pm1}]^\eta\ar[d]^{\tau_\eta}\\ 
        G(\cO)\backslash_\theta (G(F))^{\inv\circ\theta}\ar[r]^{\iota}&G[t]\backslash_\eta(G[t^{\pm1}])^{\inv\circ\eta}}
        \]
    where $\iota$ is the bijection in Theorem \ref{t:affine Matsuki},
    and $\iota_0$ is an induced bijection. 
    Moreover, the intersection of corresponding double cosets
    under $\iota_0$ is a single $G[t]-G[t^{\pm1}]^{\eta_c,\theta}$ coset.
\end{prop}
\begin{proof}
    The existence of bijection $\iota_0$
    is equivalent to that the images of $\tau_\theta$ and $\tau_\eta$
    on the set of orbits coincide under the bijection $\iota$.
    For $\epsilon=-1$,
    this follows from Proposition \ref{p:tau_theta surj epsilon=-1}
    and Proposition \ref{p:tau_eta surj epsilon=-1},
    where both $\tau_\theta,\tau_\eta$ are bijections.
    For $\epsilon=1$, 
    for any orbit $A$ in $G(\cO)\backslash G(F)/G(F)^\theta$,
    we know from Theorem \ref{t:theta spherical orbits} 
    and Lemma \ref{l:tau_theta cpt preimage}
    that $A$ contains an element $x\in G[t^{\pm1}]^{\eta_c}$
    such that 
    $\tau_\theta(x)\in((Gt^{X_*(T)}G)^{\inv\circ\theta})^{\eta_c}$.
    We get
    \[
    \iota\circ\tau_\theta(A)=\tau_\eta(G[t]xG[t^{\pm1}]^\eta).
    \]
    This shows that $\iota_0$ is a well-defined embedding.
    Note that by 
    Proposition \ref{p:spherical double cosets intersection}.(ii),
    the intersection of the orbits matched under $\iota_0$
    is a single $G[t]-G[t^{\pm1}]^{\eta_c}$ orbit.
    Thus by \cite[Theorem 1.2]{N},
    $\iota_0$ coincides with the affine Matsuki correspondence in \emph{loc. cit.}.
    By \cite[Theorem 1.2]{N} again, 
    $\iota_0$ is a bijection.
\end{proof}

In particular, we immediately obtain the following.
\begin{thm}\label{Gr}
    There is a  bijection 
    \[\Gr/G(F)^\theta\simeq \Gr/G[t,t^{-1}]^\eta\]
   between $G(F)^\theta$ and  $G[t,t^{-1}]^\eta$-orbits on the affine Grassmannian $\Gr=G(\mathcal O)\backslash G(F)$ such that the intersection of the corresponding orbit is a single $G[t,t^{-1}]^{\eta_c,\eta}$-orbit.
\end{thm}

\section{Iwahori orbits parametrization}\label{iwahori orbits para}
In this section we classify Iwahori orbits on anti-fixed points by twisted conjugation, for $\theta$ and $\eta$ respectively.
The affine Matsuki correspondence will follow from comparing the two classifications.
We also study orbit intersections and the induced correspondence
on the affine flag varieties.

\subsection{Complex Iwhoric orbits}
\subsubsection{}
Let $I$ be the Iwahori group corresponding to our fixed $B$,
with pro-unipotent radical $I(1)$ so that $I=TI(1)$.
Since $\theta_0$ preserves $T$, 
$\theta$ takes the Iwahori subgroup $I$ 
to a $W$-conjugation of $I$
that decomposes as $\theta(I)=T\theta(I(1))$.
Also, $\theta$ has induced action on the extended Weyl group
$\tW=N_{G(F)}(T(F))/T(\cO)\simeq X_*(T)\rtimes W_c$.
We fix a set of representatives $w$ of $W\simeq W_c$ in $N_{G_c}(T_c)$,
which gives rise to a set of representatives $\tw=t^\lambda w$
of $\tW$.
Denote the $\theta$-twisted conjugation by
\[
h\cdot g=hg\theta(h)^{-1}.
\]

We have decomposition
\[
G(F)^{\inv\circ\theta}=\bigsqcup_{\tw\in \tW}
(I\tw\theta(I))^{\inv\circ\theta}.
\]

If $g_1\tw g_2=\theta(g_2)^{-1}\theta(\tw)^{-1}\theta(g_1)^{-1}$,
then $\tw$ and $\theta(\tw)^{-1}$
are in the same double coset.
Write $\tw=t^\lambda w$.
We obtain
\begin{equation}\label{eq:Iwahori lambda,general}
	\tw\equiv\theta(\tw)^{-1}\in\tW\Leftrightarrow
	\theta_0(\lambda)=-w^{-1}\lambda\in X_*(T)\ \& \
	\theta_0(w)\equiv w^{-1}\in W_c.
\end{equation}

For the above $\tw$, we have 
$t_{\tw}:=(\tw\theta(\tw))^{-1}=\epsilon^\lambda(w\theta_0(w))^{-1}\in T_c$.
Denote
\begin{equation}\label{eq:Iwahori A_lambda general}
	\begin{split}
		(\tw\theta(I))^{\inv\circ\theta}
		&=\{\tw g\mid g\in\theta(I),\ 
		\Ad_{\tw}g=\theta(g)^{-1}t_{\tw}\}\\
		&=:\{\tw g\mid g\in I_\tw\}.
	\end{split}
\end{equation}

Clearly
\[
I_\tw\subset \theta(I)\cap\Ad_{\tw^{-1}}I
\]
on which $I\cap\Ad_{\tw}\theta(I)$ acts by 
$h\cdot g=(\Ad_{\tw^{-1}}h)g\theta(h)^{-1}$.

\subsubsection{}
Define
\begin{equation}\label{eq:T_lambda theta}
	T_\tw:=\{g_0\in T\mid \Ad_w g_0=\theta_0(g_0)^{-1}t_\tw\}\subset I_\tw,
\end{equation}
on which $T$ acts by $h\cdot g_0=(\Ad_{w^{-1}}h)g_0\theta_0(h)^{-1}$.

\begin{prop}\label{p:Iwahori theta}
	The natural map
	$T\backslash T_\tw\rightarrow (I\cap\Ad_{\tw}\theta(I))\backslash I_\tw$
	is a bijection.
\end{prop}
\begin{proof}
	Observe that
	\[
	\theta(I)\cap\Ad_{\tw^{-1}}I=T(\theta(I(1))\cap\Ad_{\tw^{-1}}I(1)),\quad 
	I\cap\Ad_{\tw}\theta(I)=T(I(1)\cap\Ad_{\tw}\theta(I(1))).
	\]
	The injectivity is then clear.
	
	For surjectivity, 
	let $g=g_0u\in I_\tw$, $g_0\in T$, $u\in \theta(I(1))\cap\Ad_{\tw^{-1}}I(1)$.
	They satisfy
	\[
	(\Ad_w g_0)(\Ad_\tw u)
	=(\theta_0(g_0)^{-1}t_\tw)\Ad_{t_\tw^{-1}\theta_0(g_0)}\theta(u)^{-1}.
	\]
	We get
	\[
	\Ad_\tw u=\Ad_{t_\tw^{-1}\theta_0(g_0)}\theta(u)^{-1}.
	\]
	Denote $r=t_\tw^{-1}\theta_0(g_0)\in T$.
	
	Consider $G(\cO)_k:=\ker(G(\cO)\rightarrow G(\cO/t^k))$, $k\geq1$.
	Then $\sigma:=\Ad_{\tw^{-1}r}\circ\inv\circ\theta$ acts on 
	$G(\cO)_k\cap\Ad_{\tw^{-1}}G(\cO)_k$ and $\theta(I(1))\cap\Ad_{\tw^{-1}}I(1)$.
	Since $G(\cO)_k$ is a normal subgroup of $I(1)$,
	$\sigma$ acts on unipotent group
	$U_k=\theta(I(1))\cap\Ad_{\tw^{-1}}I(1)/G(\cO)_k\cap\Ad_{\tw^{-1}}G(\cO)_k$.
	Denote $\fu_k=\mathrm{Lie}(U_k)$.
	Clearly $\sigma$ induces a linear action on $\fu_k$.
	Let $u_k$ be the image of $u$ in $U_k$.
	Since $\sigma(u)=u$, 
    $u_k=\exp(a_k)$ for $a_k\in\fu_k^\sigma$.
	Let $h_k=\Ad_{\tw g_0}\exp(-\frac{1}{2}a_k)\in I\cap\Ad_{\tw}\theta(I)$ 
	where $\exp(-\frac{1}{2}a_k)$ is taken to be any preimage in $\theta(I(1))\cap\Ad_{\tw^{-1}}I(1)$.
	We have
	\[
	h_k\cdot\tw g_0u=\tw g_0 g_k,\quad g_k\in G(\cO)_k\cap\Ad_{\tw^{-1}}G(\cO)_k.
	\]
	Repeating this process for $g_k$, 
	the action by the convergent product of the sequence of $h_k$
	maps $\tw g$ to $\tw g_0$.
	This completes the proof.
\end{proof}

Denote
\[
T_{\tw,c}=T_\tw\cap G_c
=\{g_0\in T_c\mid \Ad_w g_0=\theta_0(g_0)^{-1}t_\tw\}.
\]
Recall we take $w\in G_c$.
\begin{lem}\label{l:T_tw,c to T_tw}
	The natural map
	$T_c\backslash T_{\tw,c}\rightarrow T\backslash T_\tw$
	is a bijection.
\end{lem}
\begin{proof}
	Write $\ft=\ft_c\oplus i\ft_c$ the decomposition for $\eta_{c,0}$,
	$T\simeq T_c\times\exp(i\ft_c)$.
	Note that $\theta_0$ and $\Ad_w$ act on $\ft_c$ and $i\ft_c$.
	The proof is essentially the same as Proposition \ref{p:C_lambda to A_lambda,0}.
\end{proof}

\subsubsection{}
From the above discussion, we obtain:
\begin{thm}\label{c:theta Iwahori orbits}
We have bijections induced by $x_{\tilde w}=\tw g \mapsto g\in T_{\tw,c}$:
	\[
	I\backslash_\theta G(F)^{\inv\circ\theta}
	\simeq T(\cO)\backslash_\theta(N_{G(F)}(T(F)))^{\inv\circ\theta}
    \simeq\bigsqcup_{\theta(\tw)\equiv\tw^{-1}\in\tW}T\backslash T_{\tw}
	\simeq\bigsqcup_{\theta(\tw)\equiv\tw^{-1}\in\tW}T_c\backslash T_{\tw,c}.
	\]
     Moreover, the natural inclusion $Z_{T_{\tilde w,c}}(g)\to Z_{I}(x_{\tilde w})$
    of stabilizers
    induces an isomorphism on component groups 
$\pi_0(Z_{T_{\tilde w,c}}(g))\cong\pi_0(Z_{I}(x_{\tilde w}))$.
\end{thm}
\begin{proof}
    Same proof as Theorem \ref{t:theta spherical orbits}.
\end{proof}

\subsection{Real Iwahoric orbits}
\subsubsection{}
Recall $I_p=I\cap G[t,t^{-1}]\subset G[t]$.
Since $T$ is $\eta_0$-stable, 
$\eta(I_p)$ is a $W$-conjugation of the 
standard Iwahori subgroup in $G[t^{-1}]$. 
We have Birkhoff decomposition
\[
G[t,t^{-1}]^{\inv\circ\eta}=
\bigsqcup_{\tw\in\tW}(I_p\tw\eta(I_p))^{\inv\circ\eta},
\]
on which $I_p$ acts by $h\cdot g=hg\eta(h)^{-1}$.
Those $\tw$ appearing in the above satisfy
\[
\eta(\tw)\equiv\tw^{-1}\in\tW.
\]
Since $\eta$ has the same action on $X_*(T)$ and $W_c$ as $\theta_0$,
the above conditions on $\eta$
are the same as \eqref{eq:Iwahori lambda,general}
with the same $t_\tw=(\tw\eta(\tw))^{-1}\in T_c$.

The $I_p$-orbits on $(I_p\tw\eta(I_p))^{\inv\circ\eta}$
are in bijection with $I_p\cap\Ad_{\tw}\eta(I_p)$-orbits on $(\tw \eta(I_p))^{\inv\circ\eta}$:
\begin{equation}\label{eq:Iwahori B_lambda general}
	\begin{split}
		(\tw \eta(I_p))^{\inv\circ\eta}
		&=\{\tw g\mid g\in \eta(I_p),\ 
		\Ad_{\tw}g=\eta(g)^{-1}t_{\tw}\}\\
		&=:\{\tw g\mid g\in \eta(I_p)_\tw\}.
	\end{split}
\end{equation}

Clearly
\[
\eta(I_p)_\tw\subset \eta(I_p)\cap\Ad_{\tw^{-1}}I_p,
\]
on which $I_p\cap\Ad_{\tw}\eta(I_p)$ acts by $h\cdot g=(\Ad_{\tw^{-1}}h)g\eta(h)^{-1}$.

\subsubsection{}
Define
\begin{equation}\label{eq:T_lambda eta}
	T_\tw':=\{g_0\in T\mid \Ad_w g_0=\eta_0(g_0)^{-1}t_\tw\}\subset \eta(I_p)_\tw,
\end{equation}
on which $T$ acts by $g\cdot t=(\Ad_{w^{-1}})gt\eta_0(g)^{-1}$.

\begin{prop}\label{p:Iwahori eta}
	The natural map
	$T\backslash_{\eta_0}T_\tw'\rightarrow (I_p\cap\Ad_{\tw}\eta(I_p))\backslash_\eta \eta(I_p)_\tw$
	is a bijection.
\end{prop}
\begin{proof}
	Observe that $\eta(I_p)\cap\Ad_{\tw^{-1}}I_p=T(\eta(I_p(1))\cap\Ad_{\tw^{-1}}I_p(1))$
	where $\eta(I_p(1))\cap\Ad_{\tw^{-1}}I_p(1)$ is a finite type unipotent group,
	thus isomorphic to its Lie algebra.
	The proof is then the same as Proposition \ref{p:coset conj classes,general, real}.
\end{proof}

Since $\eta_0$ and $\theta_0$ coincide on $G_c$, we have
$T_{\tw,c}=T_\tw'\cap G_c$

\begin{lem}\label{l:T_tw,c to T_tw'}
	The natural map
	$T_c\backslash_{\theta_0} T_{\tw,c}\simeq T\backslash_{\eta_0}T_\tw'$
	is a bijection.
\end{lem}
\begin{proof}
	The proof is the same as Proposition \ref{l:T_tw,c to T_tw}.
\end{proof}

All together, we conclude
\begin{thm}\label{c:eta Iwahori orbits}
	We have bijections induced by $x_{\tilde w}=\tw g \mapsto g\in T_{\tw,c}$:
	\[
	I_p\backslash_\eta G[t^{\pm1}]^{\inv\circ\eta}
    \simeq 
    T[t]\backslash_\eta(N_{G[t^{\pm1}]}(T[t^{\pm1}]))^{\inv\circ\eta}
    \simeq
	\bigsqcup_{\theta(\tw)\equiv\tw^{-1}\in\tW}T\backslash_{\eta_0}T'_{\tw}
	\simeq
	\bigsqcup_{\theta(\tw)\equiv\tw^{-1}\in\tW}T_c\backslash_{\theta_0}T_{\tw,c}.
	\]
     Moreover, the natural inclusion $Z_{T_{\tilde w,c}}(g)\to Z_{I_p}(x_{\tilde w})$
    of stabilizers
    induces an isomorphism on component groups 
$\pi_0(Z_{T_{\tilde w,c}}(g))\cong\pi_0(Z_{I_p}(x_{\tilde w}))$.
\end{thm}

\subsection{Matsuki duality for Iwahoric orbits}

\subsubsection{Intersection of Iwahori orbits}
\begin{prop}\label{p:Iwahori orbits intersection}
    For $\theta(\tw)\equiv\tw^{-1}\in\tW$
    and $g_0\in T_{\tw,c}$,
    we have
    \[
    I\cdot_\theta \tw g_0\cap I_p\cdot_\eta \tw g_0=T_c\cdot\tw g_0.
    \]
\end{prop}
\begin{proof}
    The proof is similar and simplified from the spherical orbits.
    We still sketch it for completeness.
    Observe that the action of $\eta_{c,0}$ on $T$
    implies that $B\cap\eta_{c,0}(B)=T$,
    so that $I\cap\eta_c(I)=T$
    and $\eta(I)\cap\theta(I)=T$.

    Assume for $g_1\in I,\ g_2\in I_p$, we have
    \[
    g:=g_1\tw g_0\theta(g_1)^{-1}
     =g_2\tw g_0\eta(g_2)^{-1}.
    \]
    We get
    \[
    \eta(g_2)^{-1}\in T(\eta(I_p(1))\cap(\Ad_{\tw^{-1}}I_p(1))\theta(I))
    =T(\eta(I_p(1))\cap\Ad_{\tw^{-1}}I_p(1)).
    \]
    Here we use that 
    $\Ad_{\tw^{-1}}I_p=(\Ad_{\tw^{-1}}I_p\cap\eta(I_p))(\Ad_{\tw^{-1}}I_p\cap\theta(I_p))$.
    We obtain
    \[
    g_2\in T(\Ad_{\tw}\eta(I_p(1))\cap I_p(1)).
    \]

    Write $g_2=h_0h_1$, 
    $h_0\in T$, $h_2\in\Ad_{\tw}\eta(I_p(1))\cap I_p(1)$.
    Since $\eta_c(g)=g$, we get
    \begin{align*}
    h_1\Ad_{\tw g_0}\eta(h_1)^{-1}
    =&h_0^{-1}\eta_c(h_0)\eta_c(h_1)(\Ad_{\tw g_0}\theta(h_1)^{-1})
    (\Ad_{\tw g_0}(\theta(h_0)^{-1}\eta(h_0)))\\
    \in&T(I_p(1)\cap\eta_c(I_p(1)))=T.
    \end{align*}
    Thus $g\in \tw g_0 T=\tw T$.

    Write $g=\tw g_0 h$ where $h\in T$.
    Since $\eta_c(g)=g$ and $\eta_c(\tw g_0)=\tw g_0$,
    we get $h\in G_c$.
    Since $\theta(g)=g^{-1}$,
    $g_0 h\in T_{\tw,c}$.
    Thus $g\in\tw T_{\tw,c}$.
    From the Iwahori orbits classification,
    we know $I\cdot_\theta \tw g_0\cap\tw T_{\tw,c}=T_c\cdot\tw g_0$.
    This completes the proof.
\end{proof}

Note that $\theta$ and $\eta_c$ act on $N_{G(F)}(T(F))$.
\begin{cor}\label{c:union of Iwahori orbits intersection}
    We have decomposition
    \[
    \bigsqcup_{\substack{\theta(\tw)\equiv\tw^{-1}\in\tW\\
    \bar{g}_0\in T_c\backslash T_{\tw,c}}}
    I\cdot_\theta \tw g_0\cap I_p\cdot_\eta \tw g_0
    =((N_{G(F)}(T(F)))^{\inv\circ\theta})^{\eta_c}.
    \]
    In the above $g_0$ is any representative of the orbit $\bar{g}_0$.
\end{cor}
\begin{proof}
    By Proposition \ref{p:Iwahori orbits intersection},
    it suffices to show
    \[
    \bigsqcup_{\substack{\theta(\tw)\equiv\tw^{-1}\in\tW\\
    \bar{g}_0\in T_c\backslash T_{\tw,c}}}
    T_c\cdot\tw g_0
    =\bigsqcup_{\theta(\tw)\equiv\tw^{-1}\in\tW}\tw T_{\tw,c}
    =((N_{G(F)}(T(F)))^{\inv\circ\theta})^{\eta_c}.
    \]
    where the first equality is obvious.

    We can write $N_{G(F)}(T(F))=t^{X_*(T)}N_G(T)T(\cO)$.
    Recall any $t^\lambda$ is fixed by $\eta_c$.
    Thus 
    \[
    (N_{G(F)}(T(F)))^{\eta_c}=t^{X_*(T)}(N_G(T)T(\cO))^{\eta_c}.
    \]
    Observe $(N_G(T)T(\cO))^{\eta_c}\subset G(\cO)\cap G[t^{-1}]=G$,
    so that $(N_G(T)T(\cO))^{\eta_c}=N_G(T)^{\eta_c}$.
    We have fixed representatives $w\in G_c$ of $W_c\simeq W$,
    so that we can write $N_{G}(T)=\bigsqcup_w wT$
    and $N_G(T)^{\eta_c}=\bigsqcup_w wT_c$.
    We obtain
    \[
    ((N_{G(F)}(T(F)))^{\inv\circ\theta})^{\eta_c}
    =(\bigsqcup_{\tw}\tw T_c)^{\inv\circ\theta}
    =\bigsqcup_{\theta(\tw)\equiv\tw^{-1}\in\tW}\tw T_{\tw,c}
    \]
    as desired.
\end{proof}

\subsubsection{Matsuki duality for Iwahori orbits}
Combining the above with 
Corollary \ref{c:theta Iwahori orbits}
and Corollary \ref{c:eta Iwahori orbits}, 
we conclude
\begin{thm}\label{t:Iwahori Matsuki}
There is a unique bijection
	\[
	I_p\backslash_\eta G[t,t^{-1}]^{\inv\circ\eta}
	\simeq
	I\backslash_\theta G(F)^{\inv\circ\theta}
	\]
such that a $\theta$-twisted $I$-orbit $\mathcal O^I_X$
corresponds to
an $\eta$-twisted $I_p$-orbit $\mathcal O^I_\bR$
if and only if 
their intersection 
$\mathcal O^I_X\cap \mathcal O^I_\bR$
is a single twisted $T_c$-orbit in 
$((N_{G(F)}(T(F)))^{\inv\circ\theta})^{\eta_c}$.
 Moreover, for any $x\in \mathcal O^I_X\cap\mathcal O^I_\bR$, there is a natural isomorphism component groups of stabilizers 
    \[\pi_0(Z_{I}(x))\cong \pi_0(Z_{I_p}(x)).\]

\end{thm}

\begin{rem}\label{r:compatibility}
    In the above, observe that
    \[
    ((N_{G(F)}(T(F)))^{\inv\circ\theta})^{\eta_c}=((N_{G}(T)t^{X_*(T)})^{\inv\circ\theta})^{\eta_c}
    \subset(Gt^{X_*(T)}G)^{\inv\circ\theta})^{\eta_c}.
    \]
    Thus Theorem \ref{t:Iwahori Matsuki}
    is a refinement of Theorem \ref{t:affine Matsuki},
    where an $I$-orbit(resp. $I_p$-orbit) represented by a $T_c$-orbit $A\subset((N_{G(F)}(T(F)))^{\inv\circ\theta})^{\eta_c}$
    is sent to the $G(\cO)$-orbit(resp. $G[t]$-orbit) represented by the
    $G_c$-orbit generated by $A$ in  
    $(Gt^{X_*(T)}G)^{\inv\circ\theta})^{\eta_c}$.
\end{rem}

\begin{cor}\label{c:finite Matsuki}
    There is a unique bijection
\begin{equation}\label{G/B}
    B\backslash_{\eta_0}G^{\inv\circ\eta_0}
    \simeq B\backslash_{\theta_0}G^{\inv\circ\theta_0}  
\end{equation}
such that the intersection of the corresponding $B$-orbits
    is a single $T_c$-orbit in $((N_G(T))^{\inv\circ\theta_0})^{\eta_{c,0}}$.
\end{cor}
\begin{proof}
    Restricting Theorem \ref{t:Iwahori Matsuki}
    to those double cosets of $\tw=t^\lambda w=w$
    where $\lambda=0$,
    it only remains to show the following bijections 
    induced by natural inclusions:
    \[
    B\backslash_{\eta_0}G^{\inv\circ\eta_0}
    \simeq 
    I_p\backslash_\eta (G[t]G[t^{-1}])^{\inv\circ\eta},
    \quad
    B\backslash_{\theta_0}G^{\inv\circ\theta_0}
    \simeq
    I\backslash_\theta(G(\cO))^{\inv\circ\theta}.
    \]
    We explain the second bijection, 
    the proof for the first one is parallel.

    By Corollary \ref{c:theta Iwahori orbits},
    every $I$-orbit in $G(\cO)^{\inv\circ\theta}$
    is represented by an element in $G^{\inv\circ\theta_0}$.
    Thus the above map is surjective.
    To show injectivity,
    decompose $G=BW\theta_0(B)$.
    Elements in different cosets 
    $(Bw\theta_0(B))^{\inv\circ\theta_0}$
    belong to distinct $I$-orbits.
    By Proposition \ref{p:Iwahori theta} and its proof,
    every $B$-orbit in
    $(Bw\theta_0(B))^{\inv\circ\theta_0}$
    contains an element of $wT_w$,
    and two elements of $wT_w$ are in the same $I$-orbit
    if and only if they are in the same $T$-orbit,
    thus apriori are in the same $B$-orbit.
    This completes the proof.
\end{proof}

\begin{rem}\label{Matsuki for G/B}
The bijection in~\eqref{G/B} agrees with the 
 Matsuki duality for the  flag variety $G/B$ \cite{M}. 
\end{rem}

\subsection{Matsuki duality for affine flag varieties}\label{ss:Iwahori double cosets}
By Lemma \ref{l:tau_theta cpt preimage},
any element in 
$\tau_\theta(G(F))\cap((N_{G(F)}(T(F)))^{\inv\circ\theta})^{\eta_c}$
has a preimage in $G[t^{\pm1}]^{\eta_c}$ under $\tau_\theta$.

\begin{prop}\label{p:Iwahori double cosets intersection}
    For any $x\in G[t^{\pm1}]^{\eta_c}$ such that
    $\tau_\theta(x)\in((N_{G(F)}(T(F)))^{\inv\circ\theta})^{\eta_c}$,
    we have
    \begin{itemize}
        \item [(i)]
        $
        Ix G(F)^\theta\cap I_p x G[t^{\pm1}]^\eta\cap G[t^{\pm1}]^{\eta_c}
        =T_c x G[t^{\pm1}]^{\theta,\eta}
        $,
        
        \item [(ii)]
        $
        Ix G(F)^\theta\cap Ix G[t^{\pm1}]^\eta
        =Ix G[t^{\pm1}]^{\theta,\eta}
        $.
    \end{itemize}
\end{prop}
\begin{proof}
    (i):
    Simply replace $G(\cO)$(resp. $G[t], G_c$) in the proof of 
    Proposition \ref{p:spherical double cosets intersection}
    with $I$(resp. $I_p, T_c$)
    and use Theorem \ref{t:Iwahori Matsuki}.

    (ii):
    By Iwasawa decomposition, we have $G=G_c B$.
    Thus \eqref{eq:Iwasawa decomp loop} 
    can be rewritten as $G(F)=I G[t^{\pm1}]^{\eta_c}$
    where the product is not necessarily unique.
    Therefore any subspace $X\subset G(F)$ stable under 
    left $I$-multiplication
    can be decomposed as $X=I X^{\eta_c}$.
    By (i), we have
    \[
    Ix G(F)^\theta\cap Ix G[t^{\pm1}]^\eta
    =I(Ix G(F)^\theta\cap Ix G[t^{\pm1}]^\eta)^{\eta_c}
    =Ix G[t^{\pm1}]^{\theta,\eta}.
    \]
\end{proof}

\begin{rem}
    Replacing $I$ with $I_p$ in the above
    and evaluating at $t=1$,
    we recover the statement in finite type Matsuki correspondence
    that the intersection of the corresponding 
    $G_\bR$-orbit and $K=G^\theta$-orbit on $B\backslash G$
    is a single $K_c=G^{\eta,\theta}$-orbit.
\end{rem}

\begin{prop}\label{p:Iwahori orbits diagram}
    We have a commutative diagram of sets of $G(\cO)$ or $G[t]$ orbits:
    \[\xymatrix{
        I\backslash G(F)/G(F)^\theta\ar[d]^{\tau_\theta}\ar[r]^{\psi_0}
        & I_p\backslash G[t^{\pm1}]/G[t^{\pm1}]^\eta\ar[d]^{\tau_\eta}\\ 
        I\backslash_\theta (G(F))^{\inv\circ\theta}\ar[r]^{\psi}
        &I_p\backslash_\eta(G[t^{\pm1}])^{\inv\circ\eta}}
        \]
    where $\psi$ is the bijection in Theorem \ref{t:Iwahori Matsuki},
    and $\psi_0$ is an induced bijection. 
    Moreover, the intersection of corresponding double cosets
    under $\psi_0$ is a single $I_p-G[t^{\pm1}]^{\eta_c,\theta}$ coset.
\end{prop}
\begin{proof}
    Recall from Remark \ref{r:compatibility} that 
    the Iwahori orbits bijection $\psi$
    is a refinement of the spherical orbits bijection $\iota$.
    It follows from Proposition \ref{p:spherical orbits diagram} that
    the Iwahori orbits contained in the spherical orbits
    in the images of $\tau_\theta$ and $\tau_\eta$
    match under $\psi$.
    Thus the bijection $\psi_0$ exists
    and refines $\iota_0$ in 
    Proposition \ref{p:spherical orbits diagram}.
    The cosets intersection follows from 
    Proposition \ref{p:Iwahori double cosets intersection}.(ii).
\end{proof}

\begin{thm}\label{Fl}
    There is a  bijection 
    \[\Fl/G(F)^\theta\simeq \Fl/G[t^{\pm1}]^\eta\]
   between $G(F)^\theta$ and  $G[t^{\pm1}]^\eta$-orbits on the affine flag variety $\Fl=I\backslash G(F)$ such that the intersection of the corresponding orbit is a single $G[t^{\pm1}]^{\eta_c,\eta}$-orbit.
\end{thm}

\section{Bundles on real forms of $\mathbb P^1$}\label{Real bundles}

\subsection{Real forms of $\mathbb P^1$}
Let $\mathbb P^1$ be the complex projective line
with coordinate $t$.
Consider the conjugation $\eta:\mathbb P^1\to \mathbb P^1$
sending $t$ to $\epsilon\bar{t}^{-1}$.

\begin{defe}
Consider the scheme 
$\mathbb P^1_{\eta}$ over $\mathbb R$ that is obtained by descending 
the complex projective line $\mathbb P^1$ to $\mathbb R$ via the conjugation $c$.
\end{defe}

If $\epsilon=1$, then $\mathbb P^1_\eta\cong\mathbb P^1_\mathbb R$ is the real projective line with real points $\mathbb P^1_\mathbb R(\mathbb R)=S^1=\{t||t|=1\}$ the unit circle.
If $\epsilon=-1$, then  $\mathbb P^1_\eta\cong\widetilde{\mathbb P}^1_\mathbb R\cong\on{Proj}(\bbR[x,y,z]/(x^2+y^2+z^2))$ is the 
so called twistor-$\mathbb P^1$.
It is known that any real form of $\mathbb P^1$ is isomorphic to $\mathbb P^1_\eta$.

\subsection{Uniformization}
\subsubsection{Unramified cases}
Let $\Bun_G(\mathbb P^1)$ be the moduli stack of 
$G$-bundles on the complex projective line $\mathbb P^1$.
The involution $\eta_0$ on $G$ and the 
conjugation $c:\mathbb P^1\to\mathbb P^1$ induces a 
conjugation $\eta:\Bun_G(\mathbb P^1)\to \Bun_G(\mathbb P^1)$.

\begin{defe}
 Consider the stack 
$\Bun_G(\mathbb P^1)_{\eta}$ over $\mathbb R$ that is obtained by descending 
$\Bun_G(\mathbb P^1)$ to $\mathbb R$ via the conjugation $\eta$. 
\end{defe}

An $\mathbb R$-point of $\Bun_G(\mathbb P^1)_{\eta}$
consists of a $G$-bundle $\cE$ on $\mathbb P^1$
and an isomorphism $c:\cE\cong\eta(\cE)$ such that the following composition is the identity map:
\begin{equation}\label{c^2=id}
\cE\stackrel{c}\longrightarrow\eta(\cE)\stackrel{\eta(c)}\longrightarrow\eta^2(\cE)=\cE.
\end{equation}
If $\epsilon=1$, then
$\Bun_G(\mathbb P^1)_{\eta}=\Bun_{G_\bbR}(\mathbb P^1_\bbR)$ is the stack of $G_\bbR$-bundles on the real projective line $\bbP^1_\bbR$.
If $\epsilon=-1$, then
$\Bun_G(\mathbb P^1)_{\eta}=\Bun_{G_\bbR}(\widetilde{\mathbb P}^1_\bbR)$ is the stack of $G_\bbR$-bundles on the  twistor-$\bbP^1$.

Consider the presentation 
\begin{equation}\label{prsentaiton}
    \mathbb P^1\cong\mathbb A^1_0\bigsqcup_{\mathbb G_m}\mathbb A_\infty^1
\end{equation}
where $\mathbb A_0^1=\on{Spec}(\bC[t])$, $\mathbb A_\infty^1=\on{Spec}(\bC[t^{-1}])$, 
    and $\mathbb G_m=\mathbb A^1_0\cap\mathbb A^1_\infty=\on{Spec}(\bC[t,t^{-1}])$.
The presentation~\eqref{prsentaiton} induces a double coset presentation
\[\Bun_G(\bbP^1)(\bC)\cong G[t]\backslash G[t,t^{-1}]/G[t^{-1}].\]

\begin{prop}\label{Unif spherical}
  There is a canonical bijection
    \begin{equation}
    \Bun_G(\mathbb P^1)_\eta(\mathbb R)\cong G[t]\backslash_\eta(G[t,t^{-1}])^{\inv\circ\eta}
\end{equation}  
such that the forgetful map 
$\Bun_G(\bbP^1)_\eta(\mathbb R)\to\Bun_G(\bbP^1)(\mathbb C)$
is induced by the natural inclusion  
$G[t,t^{-1}]^{\inv\circ\eta}\to  G[t,t^{-1}]$
\end{prop}
\begin{proof}
    The conjugation $\eta$ on $\mathbb P^1$ maps  
  $\eta:\mathbb A^1_0\cong\mathbb A_\infty^1$
  and preserves $\eta:\mathbb G_m\cong\mathbb G_m$.
  Let $(\cE,c)\in\Bun_G(\mathbb P^1)_\eta(\bR)$.
  Pick a trivialization $\iota:\cE|_{\mathbb A_0^1}\cong G\times\mathbb A_0^1$. Then the composed isomorphism
  \[G\times\mathbb G_m\stackrel{\iota^{-1}}\cong\cE|_{\mathbb G_m}\stackrel{c}\cong \eta(\cE)|_{\mathbb G_m}\stackrel{\eta(\iota)}\cong G\times\mathbb G_m\]
  is given by right multiplication by an element $g\in G[t,t^{-1}]$
  and the equation~\eqref{c^2=id} implies that 
  $g\in (G[t,t^{-1}])^{\inv\circ\eta}$.
  If we change the trivialization to $\iota'$, the  
  element $g$ changes to $g'=hg\eta(h)^{-1}$ 
  where
  $h=\iota\circ(\iota')^{-1}\in G[t]$.
Since the set of trivializations of $\cE|_{\mathbb A^1_0}$
is a trivial torsor under $G[t]$, the claim follows.
\end{proof}

\subsubsection{Tamely ramified case}
Let $\Bun_{G}(\mathbb P^1,0,\infty)$ be the moduli stack of 
$G$-bundles $\cE$ on $\bbP^1$ with $B$-reductions $\cF_0$ at $0$ and $\eta_0(B)$-reductions $\cF_\infty$
at $\infty$.
The presentation~\eqref{prsentaiton}
induces a double coset presentation
\[\Bun_{G}(\mathbb P^1,0,\infty)(\bC)\cong I_p\backslash G[t,t^{-1}]/\eta(I_p).\]

Since  the conjugation $c$ on $\mathbb P^1$
permutes the points $0,\infty\in\mathbb P^1$, together with the involution $\eta_0$ on $G$, we get a conjugation $\eta:\Bun_{G}(\mathbb P^1,0,\infty)\to \Bun_{G}(\mathbb P^1,0,\infty)$.

\begin{defe}
 Consider the stack 
$\Bun_G(\mathbb P^1,0,\infty)_{\eta}$ over $\mathbb R$ that is obtained by descending 
$\Bun_G(\mathbb P^1,0,\infty)$ to $\mathbb R$ via the conjugation $\eta$. 
\end{defe}

An $\mathbb R$-point of $\Bun_G(\mathbb P^1,0,\infty)_{\eta}$ consists of $(\cE,c)\in\Bun_G(\mathbb P^1)_{\eta}$ 
together with  $B$-reductions $\cF_0$ at $0$ and $\eta_0(B)$-reductions at $\cF_\infty$
as $\infty$
such that $c(\cF_0)=\cF_\infty$.
Pick a trivialization 
$\iota:\cE|_{\mathbb A^1_0}\cong G\times\mathbb A^1_0$
such that $\iota(\cF_0)=B\subset G$.
Noting that the set of such trivializations is a trivial torsor under 
$I_p$, 
the same discussion as in the unramified case implies that:
\begin{prop}\label{unif Iwahori}
  There is a canonical bijection
    \begin{equation}
    \Bun_G(\mathbb P^1,0,\infty)_\eta(\mathbb R)\cong I_p\backslash_\eta G[t,t^{-1}]^{\inv\circ\eta}
\end{equation}  
such that the forgetful map 
$\Bun_G(\bbP^1,0,\infty)_\eta(\bbR)\to\Bun_G(\bbP^1,0,\infty)(\mathbb C)$
is induced by the natural inclusion  
$G[t,t^{-1}]^{\inv\circ\eta}\to  G[t,t^{-1}]$.
  \end{prop}

Combining the above with Proposition \ref{p:tau_eta surj epsilon=-1},
we obtain:
\begin{cor} 
Assume $\epsilon=-1$.
There are double coset presentations
\begin{align*}
    &\Bun_{G}(\mathbb P^1)_\eta(\mathbb R)\cong G[t]\backslash G[t,t^{-1}]/G[t,t^{-1}]^\eta,\\
    &\Bun_G(\mathbb P^1,0,\infty)_\eta(\mathbb R)\cong I_p\backslash G[t,t^{-1}]/G[t,t^{-1}]^\eta.
\end{align*}
\end{cor}

\subsection{Pure inner twists}\label{inner twists}
An involution $\theta_0'$ on $G$ (resp. a conjugation $\eta_0'$ on $G$) is called a \emph{pure inner twist} of $\theta_0$ (resp. $\eta_0$) if there exists $g\in G^{\on{inv}\circ\theta_0}$ (resp. $G^{\on{inv}\circ\eta_0}$) such that 
$\theta_0'=\on{Ad}_g\circ\theta_0$ (resp. $\eta_0'=\on{Ad}_g\circ\eta_0$).

\begin{lem}\label{l:pure inner twist}\mbox{}
\begin{itemize}
    \item [(i)]
    Let $\theta_0'=\Ad_g\circ\theta_0$ be a pure inner twist of $\theta_0$. There are natural bijections\newline
    $G(\mathcal O)\backslash_{\theta} G(F)^{\on{inv}\circ\theta}\simeq G(\mathcal O)\backslash_{\theta'} G(F)^{\on{inv}\circ\theta'}$
    and $I\backslash_{\theta} G(F)^{\on{inv}\circ\theta}\simeq I\backslash_{\theta'} G(F)^{\on{inv}\circ\theta'}$.
    \item [(ii)] 
    Let $\eta_0'=\Ad_g\circ\eta_0$ be a pure inner twist of $\eta_0$. There are natural bijections\newline $G[t]\backslash_{\eta}G[t,t^{-1}]^{\on{inv}\circ\eta}\simeq G[t]\backslash_{\eta'} G[t,t^{-1}]^{\on{inv}\circ\eta'_0}$
    and $I_p\backslash_{\eta}G[t,t^{-1}]^{\on{inv}\circ\eta}\simeq I_p\backslash_{\eta'}G[t,t^{-1}]^{\on{inv}\circ\eta'}$.
    \item [(iii)]
    Let $\eta_0'=\Ad_g\circ\eta_0$ be a pure inner twist of $\eta_0$. There are natural bijections\newline 
    $\Bun_{G}(\mathbb P^1)_\eta(\mathbb R)\simeq \Bun_{G}(\mathbb P^1)_{\eta'}(\mathbb R)$
    and $\Bun_G(\mathbb P^1,0,\infty)_\eta(\mathbb R)\simeq \Bun_G(\mathbb P^1,0,\infty)_{\eta'}(\mathbb R)$.
    \end{itemize}
\end{lem}
\begin{proof}
Assume $\epsilon=1$.
The map $x\to xg^{-1}$ defines a $G$-equivariant bijection $G^{\on{inv}\circ\theta_0}\cong G^{\on{inv}\circ\theta_0'}$,
and hence a $G(F)$-equivariant bijection 
$G(F)^{\on{inv}\circ\theta}=G^{\on{inv}\circ\theta_0}(F)\simeq G^{\on{inv}\circ\theta'_0}(F)=G(F)^{\on{inv}\circ\theta'}$.
Assume $\epsilon=-1$. 
Proposition \ref{p:tau_theta surj epsilon=-1} implies that $g=\gamma\theta(\gamma)^{-1}$
for some $\gamma\in G(F)$. Then 
we have $\theta'=\Ad_\gamma\circ\theta\circ\Ad_{\gamma^{-1}}$
and 
the map $x\to x\theta(\gamma)\gamma^{-1}$
defines an $G(F)$-equivariant bijection $G(F)^{\on{inv}\circ\theta}\cong G(F)^{\on{inv}\circ\theta'}$.
Part (i) follows.

For part (ii) 
follows from part (i) and  the Matsuki duality for spherical and Iwahori orbits, and
part (iii) follows from part (ii) and uniformization of real bundles in Proposition \ref{Unif spherical} and \ref{unif Iwahori}.

\end{proof}

    \subsection{Kottwitz sets}\label{Kottwitz}
    \subsubsection{Kottwitz sets $B(G,\eta)$}
Let $\Gamma=\on{Gal}(\bC/\bR)$ be the Galois group of $\bR$.
Let $W_{\bbR,\epsilon}$ be the semi-direct product $W_{\bbR,\epsilon}=\bC^\times\rtimes\Gamma$
if $\epsilon=1$ and real Weil group 
$W_{\bbR,\epsilon}=\bC^\times\sqcup j\bC^\times\subset\mathbb H^\times$
if $\epsilon=-1$ (here $\mathbb H$ is the real quaternion).
Note that we have a group extension
\[
0\rightarrow\bC^\times\rightarrow W_{\bbR,\epsilon}\stackrel{\pi}\rightarrow\Gamma\rightarrow 1\]
which is a split extension if and only if $\epsilon=1$.
The Galois group $\Gamma$ acts on $G$
via the conjugation $\eta_0$
and it induces an action of $W_{\bbR,\epsilon}$ on $G$
via the projection $\pi$.
\begin{defe}
We define $B(G,\eta)=H^1_{\on{alg}}(W_{\bbR,\epsilon},G)$ consisting of cocycles $x:W_{\bbR,\epsilon}\to G(\bC)$, such that $\lambda(\bC)=x|_{\bC^\times}:\bC^\times\to G$ for a cocharcter $\lambda:\mathbb{G}_m\to G$
over $\bC$, modulo $\eta_0$-twisted conjugation.

\end{defe}

It follows from the definition
(see, e.g., \cite[Lemma 3.1.11]{Jaburi}
or \cite[\S13.2]{Fargues})
that 
we have $B(G,\eta)\simeq\{(\lambda,g)\}/\sim$
where 
\begin{itemize}
	\item [(i)] $\lambda:\bGm\rightarrow G$ is a cocharacter over $\bC$,
	\item [(ii)] $g\in G(\bC)$ such that $g\eta_0(g)=\lambda(\epsilon)$,
	\item [(iii)] $g^{-1}\lambda(\bar t) g=\eta_0(\lambda(t))$, for $t\in\bC^\times$,
	\item [(iv)] $(\lambda,g)\sim(h\lambda h^{-1},hg\eta_0(h)^{-1})$ for $h\in G(\bC)$.
\end{itemize}

\begin{prop}\label{t:B(G_R)}
The assignment sending $(\lambda,g)$ to $\gamma(t)=\lambda(t) g$
induces a  bijection
	\[
	B(G,\eta)
	\simeq
	G[t]\backslash_\eta(G[t,t^{-1}])^{\inv\circ\eta}.
	\]
\end{prop}
\begin{proof}
We have 
\begin{align*}
    \gamma(t)\eta(\gamma(t))&=(\lambda(t)g)(\eta(\lambda(t)g))=\lambda(t)g\eta_0(\lambda(\epsilon\bar t^{-1}))\eta_0(g)\stackrel{(ii)}=\lambda(t)g\eta_0(\lambda(\epsilon\bar t^{-1}))g^{-1}\lambda(\epsilon)\\
    &\stackrel{(iii)}=\lambda(t)\lambda(\epsilon t^{-1})\lambda(\epsilon)=\lambda(1)=e.
\end{align*}
Thus we have $\gamma(t)=\lambda(t)g\in G[t,t^{-1}]^{\on{inv}\circ\eta}$.

Choose an $\eta_0$-stable maximal torus $T$.
Via conjugation, we can assume $\lambda\in X_*(T)^+$ in (i).
Note that condition (iii) is equivalent to
\begin{equation}
	g^{-1}\lambda(t) g=\eta_0(\lambda(\bar t))=\theta_0\circ\eta_{c,0}(\lambda(\bar t))=\theta_0(\lambda(t^{-1})).
\end{equation}
According to~\eqref{eq:lambda,general}, we have 
\[
w_1\lambda=-\theta_0(\lambda),
\]
and it follows that 
\[
g^{-1}\lambda g=w_1\lambda, 
\]
so that $g_0:=gw_1\in L_\lambda$.
Observe that the condition (ii) on $g$ is the same as $g_0\in B_{\lambda,0}$,
on which $L_\lambda$ acts by the $\Ad_{w_1^{-1}}\eta_0$-conjugation.
We conclude that $B(G,\eta)$ is classified by the union of
$L_\lambda\backslash B_{\lambda,0}$ 
for $\lambda=-w_1^{-1}\theta_0(\lambda)$,
which is in bijection with $G[t]\backslash_\eta(G[t,t^{-1}])^{\inv\circ\eta}$ 
by Theorem \ref{t:eta spherical orbits}. 
\end{proof}

\subsubsection{Generalizations}\label{generalizations}
Inspired by the notions of  strong real forms, rigid inner forms 
and extended Kottwitz sets
in \cite{ABV,Fargues,K},
we state the following generalizations of Matsuki bijections. 

Let $Z_{\on{tor}}$ be the torsion elements in the center of $G$.
Note that the actions of 
$\eta_0$ and $\theta_0$
on $Z_{\on{tor}}$ coincide and hence 
$Z_{\on{tor}}^{\eta_0}=Z_{\on{tor}}^{\theta_0}$ (see, e.g., \cite[Lemma 8.7]{AT}).
\begin{defe}
Fix an element $z\in Z_{\on{tor}}^{\eta_0}=Z_{\on{tor}}^{\theta_0}$.
Consider the following sets
\begin{enumerate}
    \item
$G(F)^{\on{inv}\circ\theta,z}=\{\gamma\in G(F)|\gamma\theta(\gamma)=z\}$.
\item $G[t,t^{-1}]^{\on{inv}\circ\eta,z}=\{\gamma\in G[t,t^{-1}]|\gamma\eta(\gamma)=z\}$.
\item $\Bun_G(\mathbb P^1)_{\eta,z}(\bbR)=\{(\cE,c)\}$
consisting of a $G$-bundle $\cE$ on $\mathbb P^1$
and an isomorphism $c:\cE\cong\eta(\cE)$ such that the composition is multiplication by $z$
\[
z:\cE\stackrel{c}\longrightarrow\eta(\cE)\stackrel{\eta(c)}\longrightarrow\eta^2(\cE)=\cE.
\]
\item $\Bun_G(\mathbb P^1,0,\infty)_{\eta,z}(\bbR)=\{(\cE,c,\cF_0,\cF_\infty)\}$ consisting of $(\cE,c)\in\Bun_G(\mathbb P^1)_{\eta,z}$ 
together with $B$-reductions $\cF_0$ at $0$ and $\eta_0(B)$-reduction at $\cF_\infty$
as $\infty$
such that $c(\cF_0)=\cF_\infty$.
\item
$B(G,\eta)_z=\{(\lambda,g)\}/\sim$
where 
\begin{itemize}
	\item [(i)] $\lambda:\bGm\rightarrow G$ is a cocharacter over $\bC$,
	\item [(ii)] $g\in G(\bC)$ such that $g\eta_0(g)=\lambda(\epsilon)z$,
	\item [(iii)] $g^{-1}\lambda(\bar t) g=\eta_0(\lambda(t))$, for $t\in\bC^\times$,
	\item [(iv)] $(\lambda,g)\sim(h\lambda h^{-1},hg\eta_0(h)^{-1})$ for $h\in G(\bC)$.
    \end{itemize}
\end{enumerate}
\end{defe}

Note that $G(\cO)$ (resp. $I$)
acts on $G(F)^{\on{inv}\circ\theta,z}$
by the $\theta$-twisted conjugation 
and $G[t]$ (resp. $I_p$)
acts on $G[t,t^{-1}]^{\on{inv}\circ\eta,z}$
by the $\eta$-twisted conjugation.
We have 
natural presentations
\[G[t]\backslash_\eta G[t,t^{-1}]^{\on{inv}\circ\eta,z}\cong\Bun_G(\mathbb P^1)_{\eta,z}(\bbR)\cong B(G,\eta)_z,\]
\[I_p\backslash_\eta G[t,t^{-1}]^{\on{inv}\circ\eta,z}\cong\Bun_G(\mathbb P^1,0,\infty)_{\eta,z}(\bbR),\]
and  Matsuki bijections
	\[
G(\cO)\backslash_\theta(G(F))^{\inv\circ\theta,z}
	\cong
	G[t]\backslash_\eta(G[t,t^{-1}])^{\inv\circ\eta,z},
	\]
    \[	I\backslash_\theta(G(F))^{\inv\circ\theta,z}
	\cong
I_p\backslash_\eta(G[t,t^{-1}])^{\inv\circ\eta,z}
	\]
between spherical and Iwahori orbits such that the intersections of the corresponding orbits is a single $G_c$ and $T_c$ orbit respectively.
The proof is similar to the setting when $z=1$ and we will leave the details to the readers.

\begin{rem}
    Note that  the union $\bigcup_{z\in Z_{\on{tor}}^{\eta_0}}B(G,\eta)_z$ is the extended Kottwitz set $B_e(G)$ in \cite[Section 13.6]{Fargues}.
\end{rem}

\section{Examples}\label{examples}
We describe the spherical and Iwahori orbits parametrization for the real forms 
$\GL_2(\bR)$, $\GL_1(\mathbb H)$, $\mathrm{U}(2)$,
$\mathrm U(1,1)$
of $G=\GL_2$.
Let $T\cong\mathbb G_m^2=\{t=(t_1,t_2)\}$ be the diagonal torus
and $B$ the Borel subgroup of upper triangular matrices.
We identify $X_*(T)\cong\mathbb Z^{2}=\{\lambda=(\lambda_1,\lambda_2)\}$
and $X_*(T)^+\cong\{\lambda=(\lambda_1,\lambda_2)\in\mathbb Z^2|\lambda_1\geq\lambda_2\}$.
Recall the involution 
$\theta_0$, conjugation
$\eta_0$, the compact conjugation $\eta_{c,0}$.
Let $W(G_c)=\{\begin{pmatrix}
   1& 0 \\
    0 & 1
\end{pmatrix},s=\begin{pmatrix}
   0& 1 \\
    1 & 0
\end{pmatrix}\}$ be the Weyl group of the compact real form 
$G_c=\mathrm U(2)$ and 
$\widetilde W=X_*(T)\rtimes W(G_c)=\{\tilde w=t^\lambda w\}$
be the extended affine Weyl group.
Recall $w_1\in N_{G_c}(T_c)$ such that  $\theta_0(B^-)=\Ad_{w_1}B$, $w_2=\theta_0(w_1)w_1$, 
$t_{\tw}=(\tw\theta(\tw))^{-1}$,
and the sets
\[X_*(T)^+_{\theta_0,w_1}:=\{\lambda\in X_*(T)^+|\lambda=-w_1^{-1}\theta_0(\lambda)\},\]
 \[A_{\lambda,0}=\{g\in L_\lambda\mid g_0=w_2\Ad_{w_1^{-1}}\theta_0(g)^{-1}\epsilon^\lambda\},\quad
 C_{\lambda,0}=A_{\lambda,0}\cap G_c,\]
 \[\tW_{\theta}:=\{\tw\in\tW\mid\tw\equiv\theta(\tw)^{-1}\},\]
 \[T_\tw=\{g_0\in T\mid \Ad_w g_0=\theta_0(g_0)^{-1}t_\tw\},\quad
 T_{\tw,c}=T_{\tw}\cap G_c,\]
where $L_\lambda$ is the Levi 
associated to $\lambda\in X_*(T)^+_{\theta_0,w_1}$.
Moreover, $L_\lambda$ acts on $A_{\lambda,0}$ by
$h\cdot g=hg\Ad_{w_1^{-1}}\theta_0(h^{-1})$
and $T$ acts on $T_{\tw}$ by
$h\cdot g=(\Ad_{w^{-1}}h)g\theta_0(h^{-1})$.

\subsection{$\GL_2(\bbR)$ case }
We have $\theta_0(g)={^t}g^{-1}$, $\eta_0(g)=\bar g$, $w_1=w_2=\begin{pmatrix}
   1& 0 \\
    0 & 1
\end{pmatrix}$,
and $X_*(T)^+=X_*(T)^+_{\theta_0,w_1}$.
\subsubsection{Assume $\epsilon=1$}
We shall describe the spherical and Iwahori orbits parametrization.

Spherical orbits:
Assume $\lambda_1=\lambda_2=\mu$.
Then $L_\lambda=\GL_2$, $A_{\lambda,0}=\{g\in\GL_2|g={^t}g\}\cong\GL_2/\mathrm O_2$, and 
$|L_\lambda\backslash A_{\lambda,0}|=1$
with representative $g_\lambda=e\in C_{\lambda,0}$.
Assume $\lambda_1>\lambda_2$.
Then $L_\lambda=T$, $A_{\lambda,0}=T$
and $|L_\lambda\backslash A_{\lambda,0}|=|T\backslash_{[2]} T|=1$
where $L_\lambda$ acts on $A_{\lambda,0}$
by the square actions with orbit representative 
$g_\lambda=\begin{pmatrix}
   1& 0 \\
    0 & 1
\end{pmatrix}$.

Here is the list of $G(\mathcal O)$-orbits representatives $x_\lambda=t^\lambda g_\lambda w_1^{-1}$ in
$G(F)^{\on{inv}\circ\theta}$
and the component groups
$S_\lambda=\pi_0(\on{Stab}_{G(\mathcal O)}(x_\lambda))$
of $G(\mathcal O)$-stabilizers
:
\begin{itemize}
    \item 
    For $\lambda=(\mu,\mu)$, we have 
    $x_\lambda=\begin{pmatrix}
   t^\mu& 0 \\
    0 & t^\mu
\end{pmatrix}$ and $S_\lambda=\mathbb Z/2\mathbb Z$.
\item 
    For $\lambda=(\lambda_1,\lambda_2)$, $\lambda_1>\lambda_2$, we have 
    $x_\lambda=\begin{pmatrix}
   t^{\lambda_1}& 0 \\
    0 & t^{\lambda_2}
\end{pmatrix}$ 
and $S_\lambda=\mathbb Z/2\mathbb Z\times\mathbb Z/2\mathbb Z$.
\end{itemize}

Here is the list of $\mathbb R$-points $(\mathcal E,c)\in\Bun_{G}(\mathbb P^1)_{\eta}$
and the reductive quotients 
$\on{Aut}(\mathcal E,c)_{red}$
of their automorphism groups:
\begin{itemize}
    \item 
    For $\lambda=(\mu,\mu)$,
      $\mathcal E=\mathcal O(\mu)\oplus\mathcal O(\mu)$, $c=\begin{pmatrix}
   1& 0 \\
    0 & 1
\end{pmatrix}$, and $\on{Aut}(\mathcal E,c)_{red}=\GL_2(\mathbb R)$.
\item 
For $\lambda=(\lambda_1,\lambda_2)$, $\lambda_1>\lambda_2$,
    $\mathcal E=\mathcal O(\lambda_1)\oplus\mathcal O(\lambda_2)$, $c=\begin{pmatrix}
   1& 0 \\
    0 & 1
\end{pmatrix}$, and $\on{Aut}(\mathcal E,c)_{red}=\mathbb R^\times\times\mathbb R^\times$.
\end{itemize}

Iwahori orbits: We have 
\[\{\tilde w=t^\lambda w\in\widetilde W|\theta(\tilde w)=\tilde w^{-1}\}=
\{\tilde w=t^\lambda w|w(\lambda)=\lambda \}=\{t^\lambda|\lambda\in X_*(T)\}\sqcup\{t^\lambda s|\lambda=(\mu,\mu)\}.\]
Assume $\tilde w=t^\lambda$. 
Then  $T_{\tilde w}=T$ and $|T\backslash T_{\tilde w}|=|T\backslash_{[2]}T|=1$
with representative $g_{\tilde w}=\begin{pmatrix}
   1& 0 \\
    0 & 1
\end{pmatrix}$.
Assume $\tilde w=t^\lambda s$ with $\lambda=(\mu,\mu)$.
Then $T_{\tilde w}=\{(z,z)|z\in\bC^\times\}$
and $|T\backslash T_{\tilde w}|=1$ with representative 
$g_{\tilde w}=\begin{pmatrix}
   1& 0 \\
    0 & 1\end{pmatrix}$.
    
Here is the list of $I$-orbits representatives 
$x_{\tilde w}=\tilde w g_{\tilde w}$ in
$G(F)^{\on{inv}\circ\theta}$
and the component groups
$H_\lambda=\pi_0(\on{Stab}_{I}(x_{\tilde w}))$
of $I$-stabilizers
:
\begin{itemize}
    \item 
    For $\tilde w=t^\lambda$, we have 
    $x_{\tilde w}=\begin{pmatrix}
   t^{\lambda_1}& 0 \\
    0 & t^{\lambda_2}
\end{pmatrix}$ and $S_\lambda=\mathbb Z/2\mathbb Z\times \mathbb Z/2\mathbb Z$.
\item 
    For $\tilde w=t^\lambda s$,  $\lambda=(\mu,\mu)$, we have 
    $x_{\tilde w}=\begin{pmatrix}
   0& t^{\mu} \\
    t^{\mu} & 0
\end{pmatrix}$ and $S_\lambda=e$.
\end{itemize}

Here is the list of $\mathbb R$-points $(\mathcal E,c,l_0,l_\infty)\in\Bun_{G}(\mathbb P^1,0,\infty)_{\eta}$, $l_0$ and $l_\infty=c(l_0)$ are lines in the fiber $\mathcal E$ at $0$ and $\infty$ respectively, 
and the reductive quotients 
$\on{Aut}(\mathcal E,c,l_0,l_\infty)_{red}$
of their automorphism groups:
\begin{itemize}
    \item 
    For $\tilde w=t^\lambda$,
      $\mathcal E=\mathcal O(\lambda_1)\oplus\mathcal O(\lambda_2)$, $c=\begin{pmatrix}
   1& 0 \\
    0 & 1
\end{pmatrix}$, $l_0=[1,0]$, $l_\infty=[1,0]$, and\newline 
$\on{Aut}(\mathcal E,c,l_0,l_\infty)_{red}=\bR^\times\times\bR^\times$.
\item 
For $\tilde w=t^\lambda s$, $\lambda=(\mu,\mu)$,
    $\mathcal E=\mathcal O(\mu)\oplus\mathcal O(\mu)$, $c=\begin{pmatrix}
   1& 0 \\
    0 & 1
\end{pmatrix}$, $l_0=[1,0]$, $l_\infty=[0,1]$, and $\on{Aut}(\mathcal E,c,l_0,l_\infty)_{red}\simeq\bC^\times$.
\end{itemize}    

When $\lambda=(0,0)$, the two reductions 
$l_0=[1,0],l_\infty=[1,0]$ or $[0,1]$ 
on the trivial bundle $\mathcal E=\mathcal O\oplus\mathcal O$
corresponds to the 
two $\GL_2(\bR)$-orbits on $\GL_2/B$.

\subsubsection{Assume $\epsilon=-1$}\label{sss:GL2R epsilon=-1}
We shall describe the spherical orbits parametrization.
Assume $\lambda_1=\lambda_2=\mu$.
Then $L_\lambda=\GL_2$ and 
\[A_{\lambda,0}=\{g\in\GL_2|g={^t}g(-1)^\lambda\}=
	\begin{cases}
		\{g\in\GL_2|g={^t}g\}\cong\GL_2/\mathrm O_2, \ \ \ \ \ \ \ \ \  \ \ \ \ \ \ \ \ \ \ \  \mu\in2\mathbb Z,\\
		\{g\in\GL_2|g=-{^t}g\}\cong\GL_2/\mathrm{SL}_2\cong\mathbb G_m,\ \ \ \ \ \  \ \ \mu\in2\mathbb Z+1.
	\end{cases}
\]
It follows that 
$|L_\lambda\backslash A_{\lambda,0}|=1$
with representative 
$g_\lambda=
		\begin{pmatrix}
   1& 0 \\
    0 & 1
\end{pmatrix},   \mu\in2\mathbb Z$,
and $g_\lambda=
        \begin{pmatrix}
   0& 1 \\
    -1 &0
\end{pmatrix},\newline \mu\in2\mathbb Z+1$.
Assume $\lambda_1>\lambda_2$.
Then $L_\lambda=T$ and 
\begin{align*}
A_{\lambda,0}
&=\{t\in T|t=t(-1)^\lambda\}\\
&=\{(t_1,t_2)\in \mathbb G_m^2|(t_1,t_2)=(t_1(-1)^{\lambda_1},t_2(-1)^{\lambda_2})\}=
	\begin{cases}
		T, \ \ \ \ \ \lambda_1,\lambda_2\in2\mathbb Z,\\
        \emptyset, \ \ \ \ \ \ \text{otherwise}.
	\end{cases}
\end{align*}
In the first case we have $|L_\lambda\backslash A_{\lambda,0}|=|T\backslash_{[2]} T|=1$
 with orbit representative 
$g_\lambda=\begin{pmatrix}
   1& 0 \\
    0 & 1
\end{pmatrix}$.

Here is the list of $G(\mathcal O)$-orbits 
representatives $x_\lambda=t^\lambda g_\lambda w_1^{-1}$ in
$G(F)^{\on{inv}\circ\theta}$
and the component groups
$S_\lambda=\pi_0(\on{Stab}_{G(\mathcal O)}(x_\lambda))$
of $G(\mathcal O)$-stabilizers:
\begin{itemize}
    \item 
    For $\lambda=(2\mu+1,2\mu+1)$, we have 
    $x_\lambda=\begin{pmatrix}
   0& t^{2\mu+1} \\
    -t^{2\mu+1} & 0
\end{pmatrix}$ and $S_\lambda=e$.
\item 
    For $\lambda=(2\mu,2\mu)$,  we have 
    $x_\lambda=\begin{pmatrix}
   t^{2\mu}& 0 \\
    0 & t^{2\mu}
\end{pmatrix}$ and $S_\lambda=\mathbb Z/2\mathbb Z$.
\item 
    For $\lambda=(2\mu_1,2\mu_2)$, $\mu_1>\mu_2$, we have 
    $x_\lambda=\begin{pmatrix}
   t^{2\mu_1}& 0 \\
    0 & t^{2\mu_2}
\end{pmatrix}$ and $S_\lambda=\mathbb Z/2\mathbb Z\times\mathbb Z/2\mathbb Z$.
\end{itemize}

Here is the list of $\mathbb R$-points $(\mathcal E,c)\in\Bun_{G}(\mathbb P^1)_{\eta}$
and the reductive quotients 
$\on{Aut}(\mathcal E)_{red}$
of their automorphism groups:
\begin{itemize}
    \item 
    For $\lambda=(2\mu+1,2\mu+1)$,
      $\mathcal E=\mathcal O(2\mu+1)\oplus\mathcal O(2\mu+1)$, $c=\begin{pmatrix}
   0& 1 \\
    -1 & 0
\end{pmatrix}$, and $\on{Aut}(\mathcal E)_{red}=\GL_1(\mathbb H)$.
   \item 
    For $\lambda=(2\mu,2\mu)$,
      $\mathcal E=\mathcal O(2\mu)\oplus\mathcal O(2\mu)$, $c=\begin{pmatrix}
   1& 0 \\
    0 & 1
\end{pmatrix}$, and $\on{Aut}(\mathcal E)_{red}=\GL_2(\mathbb R)$.
\item 
For $\lambda=(2\mu_1,2\mu_2)$, $\mu_1>\mu_2$,
    $\mathcal E=\mathcal O(2\mu_1)\oplus\mathcal O(2\mu_2)$, $c=\begin{pmatrix}
   1& 0 \\
    0 & 1
\end{pmatrix}$ and $\on{Aut}(\mathcal E)_{red}=\mathbb R^\times\times\mathbb R^\times$.
\end{itemize}

\begin{rem}
 The list above agrees with the well-known description of 
    vector bundles on twistor $\widetilde{\mathbb P^1_\bbR}$, see, e.g., \cite{Jaburi}.   
\end{rem}

\subsection{$\GL_1(\mathbb H)$ case}
We shall describe the spherical orbits parametrization.
We have $\theta_0(g)=\begin{pmatrix}
   0& 1 \\
    -1 & 0
\end{pmatrix}{^t}g^{-1}\begin{pmatrix}
   0& -1 \\
    1 & 0
\end{pmatrix}$, $\eta_0(g)=\begin{pmatrix}
   0& 1 \\
    -1 & 0
\end{pmatrix}\bar g\begin{pmatrix}
   0& -1 \\
    1 & 0
\end{pmatrix}$, $w_1=\begin{pmatrix}
   0& 1 \\
    -1 & 0
\end{pmatrix}$,
$w_2=\begin{pmatrix}
   -1& 0 \\
    0 & -1
\end{pmatrix}$
and $X_*(T)^+=X_*(T)^+_{\theta_0,w_1}$.
\subsubsection{Assume $\epsilon=1$}
Assume $\lambda_1=\lambda_2=\mu$.
Then $L_\lambda=\GL_2$ and 
\[A_{\lambda,0}=\{g\in\GL_2|g=-^tg\}\cong\GL_2/\on{SL}_2.\]
It follows that 
$|L_\lambda\backslash A_{\lambda,0}|=1$
with representative
$g_\lambda=
		\begin{pmatrix}
   0& 1 \\
    -1 & 0
\end{pmatrix}$.

Assume $\lambda_1>\lambda_2$.
Then $L_\lambda=T$ and 
\[A_{\lambda,0}=
\{(t_1,t_2)\in \mathbb G_m^2|(t_1,t_2)=-(t_1,t_2)\}=\emptyset.\]

Here is the list of $G(\mathcal O)$-orbits 
representatives $x_\lambda=t^\lambda g_\lambda w_1^{-1}$ in
$G(F)^{\on{inv}\circ\theta}$
and the component groups
$S_\lambda=\pi_0(\on{Stab}_{G(\mathcal O)}(x_\lambda))$
of $G(\mathcal O)$-stabilizers:
\begin{itemize}
    \item 
    For $\lambda=(\mu,\mu)$, we have 
    $x_\lambda=\begin{pmatrix}
   t^\mu& 0 \\
    0 & t^\mu
\end{pmatrix}$ and $S_\lambda=e$.
\end{itemize}

Here is the list of $\mathbb R$-points $(\mathcal E,c)\in\Bun_{G}(\mathbb P^1)_{\eta}$
and the reductive quotients 
$\on{Aut}(\mathcal E)_{red}$
of their automorphism groups:
\begin{itemize}
    \item 
    For $\lambda=(\mu,\mu)$,
      $\mathcal E=\mathcal O(\mu)\oplus\mathcal O(\mu)$, $c=\begin{pmatrix}
   1& 0 \\
    0 & 1
\end{pmatrix}$, and $\on{Aut}(\mathcal E)_{red}=\GL_1(\mathbb H)$.
\end{itemize}

\subsubsection{Assume $\epsilon=-1$}
Assume $\lambda_1=\lambda_2=\mu$.
Then $L_\lambda=\GL_2$ and 
\[A_{\lambda,0}=\{g\in\GL_2|g=-^tg(-1)^\lambda\}=\begin{cases}
		\{g\in\GL_2|g=-{^t}g\}\cong\GL_2/\mathrm{SL}_2, \ \ \  \ \ \ \ \ \   \mu\in2\mathbb Z,\\
		\{g\in\GL_2|g={^t}g\}\cong\GL_2/\mathrm{O}_2, \ \ \ \ \ \ \ \ \ \ \ \  \mu\in2\mathbb Z+1.
	\end{cases}
\]
It follows that 
$|L_\lambda\backslash A_{\lambda,0}|=1$
with representative 
$g_\lambda=
		\begin{pmatrix}
   0& 1 \\
    -1 & 0
\end{pmatrix}$ if $\mu$ even
and $g_\lambda=
		\begin{pmatrix}
   0& 1 \\
    1 & 0
\end{pmatrix}$ if $\mu$ odd.
Assume $\lambda_1>\lambda_2$.
Then $L_\lambda=T$ and 
\[A_{\lambda,0}=
\{(t_1,t_2)\in \mathbb G_m^2|(t_1,t_2)=-(t_1(-1)^{\lambda_1},t_2(-1)^{\lambda_2})\}=\begin{cases}
		 T,  \ \ \ \ \ \ \ \ \ \ \lambda_1,\lambda_2\in2\mathbb Z+1,\\
        \emptyset,  \ \ \ \ \ \ \ \ \ \ \ \ \ \ \ \ \ \ \text{otherwise}.
	\end{cases}\]
In the first case, we have $|L_\lambda\backslash A_{\lambda,0}|=|T\backslash_{[2]} T|=1$
 with orbit representative 
$g_\lambda=\begin{pmatrix}
1& 0 \\
    0 & 1
\end{pmatrix}$. 

Here is the list of $G(\mathcal O)$-orbits 
representatives $x_\lambda=t^\lambda g_\lambda w_1^{-1}$ in
$G(F)^{\on{inv}\circ\theta}$
and the component groups
$S_\lambda=\pi_0(\on{Stab}_{G(\mathcal O)}(x_\lambda))$
of $G(\mathcal O)$-stabilizers:
\begin{itemize}
    \item 
    For $\lambda=(2\mu,2\mu)$, we have 
    $x_\lambda=\begin{pmatrix}
   t^{2\mu}& 0 \\
    0 & t^{2\mu}
\end{pmatrix}$ and $S_\lambda=e$.
\item 
    For $\lambda=(2\mu+1,2\mu+1)$, we have 
    $x_\lambda=\begin{pmatrix}
t^{2\mu+1}& 0 \\
   0 & -t^{2\mu+1}
\end{pmatrix}$ and $S_\lambda=\mathbb Z/2\mathbb Z$.

\item
 For $\lambda=(2\mu_1+1,2\mu_2+1)$, $\mu_1>\mu_2$, we have 
    $x_\lambda=\begin{pmatrix}
0& -t^{2\mu_1+1} \\
 t^{2\mu_2+1}   & 0
\end{pmatrix}$ and $S_\lambda=\bZ/2\bZ\times\bZ/2\bZ$.
\end{itemize}

Here is the list of $\mathbb R$-points $(\mathcal E,c)\in\Bun_{G}(\mathbb P^1)_{\eta}$
and the reductive quotients 
$\on{Aut}(\mathcal E)_{red}$
of their automorphism groups:
\begin{itemize}
    \item 
    For $\lambda=(2\mu,2\mu)$,
      $\mathcal E=\mathcal O(2\mu)\oplus\mathcal O(2\mu)$, $c=\begin{pmatrix}
   1& 0 \\
    0 & 1
\end{pmatrix}$, and $\on{Aut}(\mathcal E)_{red}=\GL_1(\mathbb H)$.
  
\item 
 For $\lambda=(2\mu+1,2\mu+1)$,
      $\mathcal E=\mathcal O(2\mu+1)\oplus\mathcal O(2\mu+1)$, $c=\begin{pmatrix}
   1& 0 \\
    0 & -1
\end{pmatrix}$, and $\on{Aut}(\mathcal E)_{red}\simeq\GL_2(\mathbb R)$.

\item 
 For $\lambda=(2\mu_1+1,2\mu_2+1)$, $\mu_1>\mu_2$,
      $\mathcal E=\mathcal O(2\mu_1+1)\oplus\mathcal O(2\mu_2+1)$, $c=\begin{pmatrix}
   0& -1 \\
    1 & 0
\end{pmatrix}$, and $\on{Aut}(\mathcal E)_{red}=\bR^\times\times\bR^\times$.
\end{itemize}

\subsection{$\mathrm U(2)$ and $\mathrm U(1,1)$ cases}
We shall describe the spherical orbits parametrization.
Since $\mathrm U(1,1)$ is a pure inner twist of $\mathrm U(2)$, according to Lemma \ref{l:pure inner twist},
they have the same parametrization. Thus we can restrict to the case $\mathrm U(2)$ where $\theta_0(g)=g$, $\eta_0(g)=\eta_{c,0}(g)={^t}\bar g^{-1}$, $w_1=\begin{pmatrix}
		0&1\\
		1&0
	\end{pmatrix}$,  $w_2=\begin{pmatrix}
		1&0\\
		0&1
	\end{pmatrix}$, and $X_*(T)^+_{\theta_0,w_1}=\{((\mu,-\mu)|\mu\in\bZ_{\geq0})\}$.
    
\subsubsection{Assume $\epsilon=1$} 
Assume $\mu=0$. Then $L_\lambda=\GL_2$
and $A_{\lambda,0}=\{g\GL_2|g\Ad_{w_1^{-1}}g=e\}$.
We have 
$|L_\lambda\backslash A_{\lambda,0}|=|\mathrm{H}^1(\on{Ad}_{w_1^{-1}},\GL_2)|=3$ with representatives 
$g_{\lambda}=\begin{pmatrix}
		1&0\\
		0&1
	\end{pmatrix},\begin{pmatrix}
		0&1\\
		1&0
	\end{pmatrix},\begin{pmatrix}
		0&-1\\
		-1&0
	\end{pmatrix} $.
 Here $\mathrm{H}^1(\on{Ad}_{w_1^{-1}},\GL_2)$ is the Cartan cohomology of $\GL_2$ with respect to the involution 
 $\on{Ad}_{w_1^{-1}}$.

Assume $\mu>0$. Then $L_\lambda=T$,  
$A_{\lambda,0}=\{t\in T|t=\Ad_{w_1^{-1}}t^{-1}\}=\{(z,z^{-1})|z\in\bC^\times\}$, and 
$|L_\lambda\backslash A_{\lambda,0}|=| T/\{(z,z)|z\in\bC^\times\}|=1$ with representative 
$g_\lambda=\begin{pmatrix}
   1& 0 \\
    0 & 1
\end{pmatrix}$.

Here is the list of $G(\mathcal O)$-orbits 
representatives $x_\lambda=t^\lambda g_\lambda w_1^{-1}$ in
$G(F)^{\on{inv}\circ\theta}$
and the component groups
$S_\lambda=\pi_0(\on{Stab}_{G(\mathcal O)}(x_\lambda))$
of $G(\mathcal O)$-stabilizers:
\begin{itemize}
    \item 
    For $\lambda=(0,0)$, we have 
    $x_\lambda=\begin{pmatrix}
   0& 1 \\
    1 & 0
\end{pmatrix},\begin{pmatrix}
   1& 0 \\
    0 & 1
\end{pmatrix},\begin{pmatrix}
   -1& 0 \\
    0 & -1
\end{pmatrix}$ and $S_\lambda=e$.
\item 
    For $\lambda=(\mu,-\mu)$, $\mu>0$, we have 
    $x_\lambda=\begin{pmatrix}
0& t^\mu \\
   t^{-\mu} & 0
\end{pmatrix}$ and $S_\lambda=e$.
\end{itemize}

Here is the list of $\mathbb R$-points $(\mathcal E,c)\in\Bun_{G}(\mathbb P^1)_{\eta}$
and the reductive quotients 
$\on{Aut}(\mathcal E)_{red}$
of their automorphism groups:
\begin{itemize}
    \item 
    For $\lambda=(0,0)$,
      $\mathcal E=\mathcal O\oplus\mathcal O$, $c=\begin{pmatrix}
   0& 1 \\
    1 & 0
\end{pmatrix},\begin{pmatrix}
   1& 0 \\
    0 & 1
\end{pmatrix},\begin{pmatrix}
   -1& 0 \\
    0 & -1
\end{pmatrix}$, and $\on{Aut}(\mathcal E)_{red}=\mathrm{U}(1,1), \mathrm{U}(2,0), \mathrm{U}(0,2)$ are the purely inner forms of $\mathrm U(1,1)$.
  
\item 
 For $\lambda=(\mu,-\mu)$, $\mu>0$,
      $\mathcal E=\mathcal O(\mu)\oplus\mathcal O(-\mu)$, $c=\begin{pmatrix}
   0& 1 \\
    1 & 0
\end{pmatrix}$, and $\on{Aut}(\mathcal E)_{red}=\{(z,\bar z)|z\in\bC^\times\}\subset T$.
\end{itemize}

\subsubsection{Assume $\epsilon=-1$} 
Assume $\mu=0$. Then $L_\lambda=\GL_2$
and $A_{\lambda,0}=\{g\in\GL_2|g\Ad_{w_1^{-1}}g=e\}$.
We have 
$|L_\lambda\backslash A_{\lambda,0}|=|\mathrm{H}^1(\on{Ad}_{w_1^{-1}},\GL_2)|=3$ with representatives 
$g_{\lambda}=\begin{pmatrix}
		1&0\\
		0&1
	\end{pmatrix},\begin{pmatrix}
		0&1\\
		1&0
	\end{pmatrix},\begin{pmatrix}
		0&-1\\
		-1&0
	\end{pmatrix} $.
 Here $H^1(\on{Ad}_{w_1^{-1}},\GL_2)$ is the Cartan cohomology of $\GL_2$ with respect to the involution 
 $\on{Ad}_{w_1^{-1}}$.
 
Assume $\mu>0$. Then $L_\lambda=T$,  
$A_{\lambda,0}=\{t\in T|t=\Ad_{w_1^{-1}}t^{-1}(-1)^\mu\}=\{(z,(-1)^\mu z^{-1})|z\in\bC^\times\}$, and 
$|L_\lambda\backslash A_{\lambda,0}|=| T/\{(z,z)|z\in\bC^\times\}|=1$ with representative 
$g_\lambda=\begin{pmatrix}
   1& 0 \\
    0 & (-1)^\mu
\end{pmatrix}$.

Here is the list of $G(\mathcal O)$-orbits 
representatives $x_\lambda=t^\lambda g_\lambda w_1^{-1}$ in
$G(F)^{\on{inv}\circ\theta}$
and the component groups
$S_\lambda=\pi_0(\on{Stab}_{G(\mathcal O)}(x_\lambda))$
of $G(\mathcal O)$-stabilizers:
\begin{itemize}
    \item 
    For $\lambda=(0,0)$, we have 
    $x_\lambda=\begin{pmatrix}
   0& 1 \\
    1 & 0
\end{pmatrix},\begin{pmatrix}
   1& 0 \\
    0 & 1
\end{pmatrix},\begin{pmatrix}
   -1& 0 \\
    0 & -1
\end{pmatrix}$ and $S_\lambda=e$.

\item 
    For $\lambda=(\mu,-\mu)$, $\mu>0$, we have 
    $x_\lambda=\begin{pmatrix}
0& t^\mu \\
   (-t)^{-\mu} & 0
\end{pmatrix}$ and $S_\lambda=e$.
\end{itemize}

Here is the list of $\mathbb R$-points $(\mathcal E,c)\in\Bun_{G}(\mathbb P^1)_{\eta}$
and the reductive quotients 
$\on{Aut}(\mathcal E)_{red}$
of their automorphism groups:
\begin{itemize}
    \item 
    For $\lambda=(0,0)$,
      $\mathcal E=\mathcal O\oplus\mathcal O$, $c=\begin{pmatrix}
   0& 1 \\
    1 & 0
\end{pmatrix},\begin{pmatrix}
   1& 0 \\
    0 & 1
\end{pmatrix},\begin{pmatrix}
   -1& 0 \\
    0 & -1
\end{pmatrix}$, and $\on{Aut}(\mathcal E)_{red}=\mathrm{U}(1,1), \mathrm{U}(2,0), \mathrm{U}(0,2)$ are the pure inner forms of $\mathrm U(2)$.
  
\item 
 For $\lambda=(\mu,-\mu)$, $\mu>0$,
      $\mathcal E=\mathcal O(\mu)\oplus\mathcal O(-\mu)$, $c=\begin{pmatrix}
   0& 1 \\
    (-1)^\mu & 0
\end{pmatrix}$, and $\on{Aut}(\mathcal E)_{red}=\{(z,\bar z)|z\in\bC^\times\}\subset T$.
\end{itemize}
    
\begin{rem}
Note that 
$G(F)^{\theta}=\GL_2(\!(t^2)\!)$. 
Under the bijection 
$\tau_{\theta}:\GL_2(F)/\GL_2(\! t^2)\!)\cong\GL_2(F)^{\on{inv}\circ\theta}$
sending $\gamma\mapsto\gamma\theta(\gamma)^{-1}$, 
the elements
$\begin{pmatrix}
	t&0\\
	0&1
\end{pmatrix},\begin{pmatrix}
	1&0\\
	0&1
\end{pmatrix},\begin{pmatrix}
	t&0\\
	0&t
\end{pmatrix},\begin{pmatrix}
	1&t\\
	0&t^{\mu+1}
\end{pmatrix}$
are mapped into the $\GL_2(\mathcal O)$-orbits of 
$\begin{pmatrix}
	-1&0\\
	0&1
\end{pmatrix},\begin{pmatrix}
	1&0\\
	0&1
\end{pmatrix},\begin{pmatrix}
	-1&0\\
	0&-1
\end{pmatrix},\begin{pmatrix}
	0&t^{\mu}\\
	(-t)^{-\mu}&0
\end{pmatrix}$
 respectively,
 where the first matrix is in the same orbit as 
 $\begin{pmatrix}
   0& 1 \\
    1 & 0
\end{pmatrix}$.
Then the above description of orbits 
implies that 
   the double cosets $\GL_2(\mathcal O)\backslash\GL_2(F)/\GL_2(\!(t^2)\!)$
	are represented by  
    \[\begin{pmatrix}
	t&0\\
	0&1
\end{pmatrix},\begin{pmatrix}
	1&0\\
	0&1
\end{pmatrix},\begin{pmatrix}
	t&0\\
	0&t
\end{pmatrix},\begin{pmatrix}
	1&t\\
	0&t^{\mu+1}
\end{pmatrix}, \mu>0.\] 
    In fact, the same computation
    shows that 
    the double cosets $\GL_n(\mathcal O)\backslash\GL_n(F)/\GL_n(\!(t^2)\!)$
    for $\GL_n$
	are represented by 1-by-1 blocks of $1,t$ together with
    2-by-2 blocks
    $\begin{pmatrix}
	1&t\\
	0&t^{\mu+1}
\end{pmatrix}$, $\mu>0$, providing a different argument of  \cite[Theorem 2.1]{Jin}.
\end{rem}

\end{document}